\documentclass[reqno]{amsart}

\usepackage[all]{xy}
\usepackage{graphicx}

\usepackage{amssymb}
\usepackage{amsmath}
\usepackage{mathrsfs}
\usepackage{epsfig}
\usepackage{amscd}
\usepackage{graphicx, color}

\def\E{\ifmmode{\mathbb E}\else{$\mathbb E$}\fi} 
\def\N{\ifmmode{\mathbb N}\else{$\mathbb N$}\fi} 
\def\R{\ifmmode{\mathbb R}\else{$\mathbb R$}\fi} 
\def\Q{\ifmmode{\mathbb Q}\else{$\mathbb Q$}\fi} 
\def\C{\ifmmode{\mathbb C}\else{$\mathbb C$}\fi} 
\def\H{\ifmmode{\mathbb H}\else{$\mathbb H$}\fi} 
\def\Z{\ifmmode{\mathbb Z}\else{$\mathbb Z$}\fi} 
\def\P{\ifmmode{\mathbb P}\else{$\mathbb P$}\fi} 
\def\T{\ifmmode{\mathbb T}\else{$\mathbb T$}\fi} 
\def\SS{\ifmmode{\mathbb S}\else{$\mathbb S$}\fi} 
\def\DD{\ifmmode{\mathbb D}\else{$\mathbb D$}\fi} 

\newcommand{\del}{\partial}
\newcommand{\Cont}{{\operatorname{Cont}}}
\newcommand{\Hom}{{\operatorname{Hom}}}

\newcommand{\ben}{\begin{enumerate}}
\newcommand{\een}{\end{enumerate}}
\newcommand{\be}{\begin{equation}}
\newcommand{\ee}{\end{equation}}
\newcommand{\bea}{\begin{eqnarray}}
\newcommand{\eea}{\end{eqnarray}}
\newcommand{\beastar}{\begin{eqnarray*}}
\newcommand{\eeastar}{\end{eqnarray*}}

\theoremstyle{theorem}
\newtheorem{thm}{Theorem}[section]
\newtheorem{cor}[thm]{Corollary}
\newtheorem{lem}[thm]{Lemma}
\newtheorem{prop}[thm]{Proposition}

\theoremstyle{definition}
\newtheorem{defn}[thm]{Definition}
\newtheorem{rem}[thm]{Remark}
\newtheorem{ques}[thm]{Question}

\newtheorem{exm}[thm]{Example}

\newtheorem{notation}[thm]{\rm\bfseries{Notation}}

\newtheorem*{thm*}{Theorem}

\numberwithin{equation}{section}

\def\R{{\mathbb R}}

\def\E{{\mathbb E}}
\def\Z{{\mathbb Z}}
\def\C{{\mathbb C}}
\def\R{{\mathbb R}}
\def\P{{\mathbb P}}

\def\N{{\mathbb N}}

\def\11{{\mathbb I}}

\def\delbar{{\overline \partial}}

\def\id{\text{\rm id}}

\def\C{\mathbb{C}}
\def\Z{\mathbb{Z}}

\def\T{\mathbb{T}}

\def\Q{\mathbb{Q}}

\def\E{\ifmmode{\mathbb E}\else{$\mathbb E$}\fi} 
\def\N{\ifmmode{\mathbb N}\else{$\mathbb N$}\fi} 
\def\R{\ifmmode{\mathbb R}\else{$\mathbb R$}\fi} 
\def\Q{\ifmmode{\mathbb Q}\else{$\mathbb Q$}\fi} 
\def\C{\ifmmode{\mathbb C}\else{$\mathbb C$}\fi} 
\def\H{\ifmmode{\mathbb H}\else{$\mathbb H$}\fi} 
\def\Z{\ifmmode{\mathbb Z}\else{$\mathbb Z$}\fi} 
\def\P{\ifmmode{\mathbb P}\else{$\mathbb P$}\fi} 
\def\SS{\ifmmode{\mathbb S}\else{$\mathbb S$}\fi} 
\def\DD{\ifmmode{\mathbb D}\else{$\mathbb D$}\fi} 

\def\R{{\mathbb R}}

\def\E{{\mathbb E}}
\def\Z{{\mathbb Z}}
\def\C{{\mathbb C}}
\def\R{{\mathbb R}}

\def\N{{\mathbb N}}

\def\delbar{{\overline \partial}}

\def\CL{{\mathcal L}}

\def\CO{{\mathcal O}}

\def\CU{{\mathcal U}}

\def\darr#1{\raise1.5ex\hbox{$\leftrightarrow$}
\mkern-16.5mu #1}

\def\roughly#1{\raise.3ex\hbox{$#1$\kern-.75em
\lower1ex\hbox{$\sim$}}}

\def\opname#1{\mathop{\kern0pt{\rm #1}}\nolimits}

\def\vol{\opname{vol}}

\def\Reeb{\operatorname{Reeb}}
\def\Aut{\operatorname{Aut}}
\def\coker{\operatorname{Coker}}

\def\Cont{\operatorname{Cont}}

\def\Diff{\operatorname{Diff}}

\def\Sym{\operatorname{Sym}}

\DeclareFontFamily{U}{MnSymbolC}{}
\DeclareSymbolFont{MnSyC}{U}{MnSymbolC}{m}{n}
\DeclareFontShape{U}{MnSymbolC}{m}{n}{
    <-6>  MnSymbolC5
   <6-7>  MnSymbolC6
   <7-8>  MnSymbolC7
   <8-9>  MnSymbolC8
   <9-10> MnSymbolC9
  <10-12> MnSymbolC10
  <12->   MnSymbolC12}{}
\DeclareMathSymbol{\intprod}{\mathbin}{MnSyC}{'270}

\begin{document}

\quad \vskip1.375truein

\title[Strict contactomorphisms]{Strict contactomorphisms are scarce}

\author{Yong-Geun Oh, Yasha Savelyev}
\address{Center for Geometry and Physics, Institute for Basic Science (IBS),
79 Jigok-ro 127beon-gil, Nam-gu, Pohang, Gyeongbuk, Korea 37673,
\& POSTECH, Gyeongbuk, Korea}
\email{yongoh1@postech.ac.kr}
\address{University of Colima, Cuicbas, Mexico}
\email{yasha.savelyev@gmail.com}

\thanks{The first-named author's work is supported by the IBS project \# IBS-R003-D1}

\date{April 2025}

\begin{abstract} The notion of \emph{non-projectible contact forms} on a given compact manifold $M$
is introduced by the first-named author in \cite{oh:nonprojectible}, the set of which 
he also shows is a residual subset of the set of (coorientable) contact forms  both in the case with
a fixed contact structure and in the case without it.
In this paper, we prove that for any such non-projectible contact form $\lambda$
the set, denoted by $\Cont^{\text{\rm st}}(M,\lambda)$, consisting of strict contactomorphisms of $\lambda$ is a 
a countable disjoint union of real lines $\R$, one for each connected component.
\end{abstract}

\maketitle

\tableofcontents

\section{Introduction}

Let $(M,\xi)$ be a contact manifold. Any coorientable contact structure $\xi$ is defined by 
a contact form $\lambda$ $\xi = \ker \lambda$. 
In contact geometry or in contact dynamics, diffeomorphisms $\psi$ \emph{strictly} preserving 
the contact form, not just preserving  the associated contact distribution, are often considered. 
In the former case $\psi$ satisfies 
the equation 
\be\label{eq:strict-contact}
\psi^*\lambda = \lambda,
\ee
and in the latter case $\psi$ satisfies
\be\label{eq:contact}
d\psi(\xi) = \xi
\ee
for the associated contact distribution $\xi: =\ker \lambda$.
In the first sight, general consideration of the equation \eqref{eq:strict-contact}
looks as reasonable as that of \eqref{eq:contact} if one regards the set of
such $\psi$ as the automorphism group $\Aut(\lambda)$ of the geometric 
structure defined by a contact form $\lambda$. However in the deeper level, 
\eqref{eq:strict-contact} not only requires $\psi$ to satisfy \eqref{eq:contact}
for the associated contact distribution $\xi: =\ker \lambda$ but also to preserve the associated
Reeb vector field, i.e., to satisfy the additional equation
 $$
 \psi_*R_\lambda = R_\lambda
 $$
 which forces the contactomorphism to descend to the quotient space $M/\sim$, the set of
 Reeb trajectories of the leaf space of Reeb foliations. 
 This quotient space is commonly non-Hausdorff, and the deformation problem of this kind
is closely tied to the topology and dynamics of the associated foliation.
We refer readers to  \cite{oh:nonprojectible} and Theorem \ref{thm:nonprojectible-intro} below
 to see how the latter affects  the deformation of  Reeb foliations. (We also refer to \cite{oh-park:coisotropic}
  for the study of general defomation problem of \emph{presymplectic manifolds} noticing that
 contact manifold can be regarded as a special case of of a (exact) presymplectic manifold $(M,d\lambda)$.
 Also see \cite{LOTV} in which deformation of contact structures is studied as a special case of 
 Jacobi structures.)
 
In hindsight, this `overdeterminedness' is the main reason why
such a strict contact diffeomorphism is hard to construct and has been scarce beyond the universally 
present Reeb flows $\phi_{R_\lambda}^t$. To the best knowledge of the present authors,
all common examples of strict contactomorphisms in the literature come either from 
the geodesic flows of Riemannian  manifolds or from the Hopf flows of  the prequantization 
of integral symplectic manifolds.

\subsection{Statement of the main result}

Let $(M,\xi)$ be a contact manifold. A contactomorphism $\psi$ is a diffeomorphism 
satisfying $d\psi(\xi) \subset \xi$ pointwise. Then a strict contactomorphism 
with respect to $\lambda$ is the one satisfying $\psi^*\lambda = \lambda$. 
As indicated above, deforming the pairs $(\lambda,\psi)$ rather than $\psi$ alone with fixed $\lambda$, 
is necessary for the problem of our interest. We call any pair
$(\lambda,\psi)$ a \emph{strict contact pair} if it satisfies \eqref{eq:strict-contact}.
We denote by
$$
\mathfrak{C}(M)
$$
the set of contact forms, i.e., nondegenerate one-forms, and by
$$
\mathfrak{C}(M,\xi)
$$
that of contact forms $\lambda$ satisfying $\ker \lambda = \xi$ with a \emph{fixed} contact structure $\xi$.

The equation $\psi^*\lambda = \lambda$ 
has a trivial symmetry of rescaling of contact form $\lambda \mapsto C \lambda$ for any
positive constant $C > 0$. Because of this, one should consider the quotient space
\be\label{eq:R-quotient}
\mathfrak{C}(M)/\R_+ \quad \text{or }\quad \mathfrak{C}(M,\xi)/\R_+
\ee
by modding out this conformal rescaling. 
Proving the following is a key step towards the proof of our main result.

\begin{thm}[Theorem \ref{thm:Phi-submersion}]\label{thm:Phisubmersion-intro}
The set of strict contact pairs, i.e., $(\lambda,\psi)$ satisfying the equation
\eqref{eq:strict-contact} is a (Frechet) smooth submanifold of 
$$
\mathfrak{Diff}(M,\mu) = \bigcup_{\lambda \in \mathfrak{C}(M)/\sim} 
\{\lambda\} \times \Diff(M, \mu_\lambda)
$$
where $\Diff(M, \mu_\lambda)$ is the set of volume-preserving diffeomorphisms 
of the volume form $\mu_\lambda = \lambda \wedge (d\lambda)^n$.
\end{thm}
(We refer readers to Theorem \ref{thm:Phisubmersion2-intro} below for the more precise statement.)

A natural class of $\lambda$-strict contactomorphism is the set of $\lambda$ Reeb flows
for any $\lambda$. 
\begin{exm} Two well-known classes of strict contactomorphisms arise in the following ways:
\begin{enumerate} 
\item (Geodesic flows) 
On the unit sphere bundle $U^*N: = U(T^*N)$ with respect to the canonical
contact form associated to a Riemannian metric $g$, all geodesic flows of $U^*N$ are
strict contact. The associated contact form is the Liouville form  $\theta = p\, dq$
restricted to $U^*N: = \{ \beta \in T^*N \mid |\beta|_{g^*} = 1\}$
where $g^*$ is the dual metric on $T^*N$ of the metric $g$ on $TN$.
\item (Hopf flows of the prequantization) Let $(M)$ be an integral symplectic manifold and
$(Q,\lambda)$ be its prequantization principal $S^1$-bundle equipped with the connection 
one-form $\lambda$ on the circle bundle $\pi: Q \to M$ satisfying $\pi^*\omega = d\lambda$. 
Then the principal $S^1$-action is a family of strict contactomorphisms of $(Q,\lambda)$.
\end{enumerate}
\end{exm}
Both flows are however all Reeb flows of the associated contact forms. 
Beyond the Reeb flows, the isometries of the standard contact structure of $S^{2n+1} \subset \C^{n+1}$ are such examples.

However the following general construction of strict contactomorphisms is well-known to the experts.
(See \cite{casals-spacil} for example.)
\begin{lem}\label{eq:strict-contact-H} Suppose $H = H (x)$ satisfies $R_\lambda[H] = 0$.
Then the Hamiltonian flow $\psi_H^t$ is a strict contact flow and vice versa.
\end{lem}

The following definition is introduced by the first named-author  \cite{oh:nonprojectible}.

\begin{defn}[Nonprojectible contact form]\label{defn:nonprojectible} 
We call a contact form \emph{non-projectible} if it satisfies
$$
\left\{H \in C^\infty(M) \, \Big|\,  R_\lambda[H] = 0, \, \int_M H\, d\mu_\lambda = 0 \right\} = \{0\}.
$$
We denote the subset thereof by
$$
\mathfrak{C}^{\text{\rm np}} (M) \subset \mathfrak{C}(M).
$$
\end{defn}
An immediate consequence of the defining property of non-projectible contact form is that 
there is no nontrivial \emph{autonomous strict contact isotopy} other than Reeb flows. A priori
it does not rule out the possibility of existence of \emph{isolated} (modulo Reeb flows) 
strict contactomorphisms which is contact isotopic to the identity via a strict contact isotopy
generated by \emph{time-dependent Hamiltonians} $H = H(t,x)$. 

We denote by $\Cont(M,\xi)$ the set of contactomorphisms
of $\xi$ and by $\Cont_0(M,\xi)$ that of contactomorphisms contact isotopic to
the identity, and by
\be
\Cont^{\text{\rm st}}(M,\lambda) 
\ee
the set of strict contactomorphisms.

The following genericity of projectible contact forms 
is also proved in \cite{oh:nonprojectible} both on the \emph{big phase space} and on the
\emph{small phase space} borrowing the terminology employed in \cite{do-oh:reduction}.

\begin{thm}[Theorem 1.6 \& 1.10 \cite{oh:nonprojectible}] \label{thm:nonprojectible-intro} 
\begin{enumerate}
\item There is a residual subset of $\mathfrak{C}(M)$ consisting of non-projectible contact forms.
\item There is a residual subset of $\mathfrak{C}(M,\xi)$ consisting of non-projectible contact forms
$\lambda$ satisfying $\ker \lambda = \xi$ for given contact structure $\xi$.
\end{enumerate}
\end{thm}

The main purpose of the present paper is to prove that the presence of a strict
contactomorphism other than Reeb flows is indeed accidental, the precise formulation 
of which is now in order.

We have the canonical  subset $\text{\rm Reeb}(M,\lambda)$ of $\Cont_0^{\text{\rm st}}(M)$
that induces a (nonlinear) exact sequence
$$
\{1\} \to \text{\rm Reeb}(M,\lambda) \to \Cont_0^{\text{\rm st}}(M,\lambda) 
\to \Cont_0^{\text{\rm st}}(M,\lambda)/\text{\rm Reeb}(M,\lambda) \to \{[1]\}.
$$
It is easy to check that $\text{\rm Reeb}(M,\lambda)$
 is a normal subgroup of $\Cont_0^{\text{\rm st}}(M,\lambda)$.
A natural question arising at this point is how big the quotient group
\be\label{eq:quotient-group}
\Cont_0^{\text{\rm st}}(M,\lambda)/\text{\rm Reeb}(M,\lambda) 
\ee
is, the study of which is the main theme of the present paper.

\begin{thm}\label{thm:nostrict-intro} For any non-projectible contact form, the quotient
\eqref{eq:quotient-group} is a zero-dimensional manifold with a countable number of
elements, or equivalently the set $\Cont_0^{\text{\rm st}}(M, \lambda)$ is a countable union of $\R$-orbits, 
one for each connected component generated by the left action of the subgroup
$\Reeb(M,\lambda) \cong \R$.
\end{thm}
This theorem now rules out the existence of \emph{nonautonomous strict contact isotopy} too.
It however does not rule out the possibility that  
there may be a \emph{time-dependent Hamiltonian} $H = H(t,x)$ whose
 time-one map $\psi_H^1$ is a strict contactomorphism but not the whole flow,
i.e., whose conformal exponent function $g_{(\psi_H^t;\lambda)} \not \equiv 0$ although
$g_{(\psi_H^1;\lambda)} = 0$.

A similar statement is proved by Casals and Spacil \cite[Proposition 12]{casals-spacil} \emph{under
the assumption that the contact form $\lambda$ admits a dense Reeb orbit}. Obviously 
their dynamical hypothesis is not a generic statement. Our theorem
first replaces the dynamical hypothesis by the cohomological hypothesis of \cite{oh:nonprojectible}
and then makes the statement one that holds for a generic contact form. 
The present theorem then leads to a natural problem of classification of the contact forms
that admits a non-trivial automorphism group other than Reeb flows.

Combining the two theorems, we now establish the following scarcity of strict contactomophisms
as stated in the title of the present paper.
\begin{cor} \label{cor:scarcity} There exists a residual subset of $\mathfrak{C}(M)$
 on a given orientable manifold such that
the group $\Cont^{\text{\rm st}}(M,\lambda)$ consisting of strict contactomorphisms of $\lambda$
is a countable disjoint union of real lines $\R$, 
one for each connected component generated by the action of Reeb flows. The same
statement also holds inside $\mathfrak{C}(M,\xi)$.
\end{cor}

The following is still an interesting open question.

\begin{ques} Is it possible to completely eliminate all strict contactomorphisms 
other than the Reeb flows for another smaller residual subset of 
contactomorphisms?
\end{ques}

\subsection{Strategy of the proof}

We now explain our strategy of the proof of the main theorem. 
We first recall a basic 
equivalence relation of two contact structures and contact forms.

\begin{defn} Let $(M_1,\xi_1)$, $(M_2,\xi_2)$ be two contact manifolds.
\begin{enumerate}
\item  A diffeomorphism  $\psi: M_1 \to M_2$ are said to be a contact diffeomorphism, it
$d\psi(\xi_1) \subset \xi_2$.
\item We say $(M_1,\xi_1)$, $(M_2,\xi_2)$ are diffeomorphic if there exists a contact
diffeomorphism between the two.
\end{enumerate}
\end{defn}

For coorientable contact structures equipped with the associated contact forms $\lambda_1$, $\lambda_2$,
this equivalence relation can be expressed as
$$
\psi^*\lambda_2 = f \lambda_1
$$
for some nowhere vanishing function $f$ on $M_1$.

\begin{defn}\label{defn:contacotmorphism-xi} Let $(M,\xi)$ be a contact manifold.
\begin{enumerate}
\item 
A self-diffeomorphism $\psi: M \to M$ is called a contactomorphism if $d\psi (\xi) \subset \xi$.
\item An isotopy $\psi_t$ of contactomorphisms of $\xi$ is called a contact isotopy.
We say two contactomorphisms $\psi_1$ and $\psi_2$ are contact isotopic if 
there is a contact isotopy between the two.
\end{enumerate}
\end{defn}

For the purpose of studying the structure of the quotient space
$$
\Cont_0^{\text{\rm st}}(M,\lambda)/\text{\rm Reeb}(M,\lambda),
$$
 we define $\mathfrak{C}_1(M) \subset \Omega^1(M)$ to be the set of 
nondegenerate one-forms, i.e., the set of contact forms. Then 
we consider the equation
\be\label{eq:strict-contact-intro}
\psi^*\lambda -\lambda =0
\ee
for the pair $(\lambda,\psi)$ which is defined originally on the product
$$
\mathfrak{C}_1(M) \times \Diff(M).
$$
In practice, we need to make the domain slightly smaller to properly handle the 
conformal invariance of the contact structure as explained below.

\begin{rem} The current situation is similar to the proof of the fact that
the group $O(n)$ (or $U(n)$) carries a smooth structure or a submanifold of 
$GL(n,\R)$. Recall that to apply the preimage value theorem we regard the map
$$
A \mapsto A^tA
$$
as a map from the set $M^{n\times n}(\R)$ of $n\times n$ matrices to the set $\Sym(\R^n)$ of symmetric matrices, 
\emph{not } to $M^{n\times n}(\R)$.
In our current situation, we need to find the correct domain and codomain of the study of 
the strict contact pair equation $\psi^*\lambda - \lambda = 0$.
\end{rem}
For this purpose, we will find all `hidden symmetry' of the equation
$\psi^*\lambda - \lambda = 0$ and identify the correct domain and the codomain
of the relevant equation as in the example mentioned in this remark.
We have already modded out the obvious rescaling symmetry of the strict contact pairs $(\lambda,\psi)
\mapsto (C\lambda,\psi)$ by considering the quotient \eqref{eq:R-quotient}.

The first step of the reduction involves a different characterization of strict contact pairs
by splitting the equation $\psi^*\lambda = \lambda$ into two parts:
\begin{enumerate}
\item $(\lambda,\psi)$ is a contact pair, and
\item the associated conformal exponent function $g_{(\psi;\lambda)}$ vanishes, or equivalently 
$e^{g_{(\psi;\lambda)}} \equiv 1$.
\end{enumerate}
A key observation, which has been made by the first-named author in
the joint work \cite{do-oh:reduction} with Hyun-Seok Do,
 is that the second property stated above can be restated in terms of 
the volume form $\mu_\lambda: = \lambda \wedge (d\lambda)^n$ instead of the contact form
$\lambda$ itself. This restatement vastly reduces the number of equations from $2n+1$ to
$1$ formulating  the statement (2) above: 
Recall that the dual space of a vector space $\R^{2n+1}$ is of dimension $2n+1$, while the 
top exterior power $\Lambda^{2n+1}(V^*)$ has dimension 1.  More detailed explanation of
this reduction is now in order.

By a slight abuse of notation, we also denote by
 $\mu_\lambda$  the measure induced by the $(2n+1)$-form.
 We recall the following definition introduced in \cite{do-oh:reduction}.
\begin{defn}[$\lambda$-incompressible diffeomorphisms] Fixing a coorientation of $\xi$ and
denote by 
$\mathfrak{C}^+(M,\xi)$ the set of positive contact forms of $\xi$.
Let $\lambda \in \mathfrak{C}^+(M,\xi)$.
We call a $\mu_\lambda$-preserving diffeomorphisms a 
\emph{$\lambda$-incompressible diffeomorphism}, and denote the set by
$$
\Diff(M,\mu_\lambda) = \{\varphi \in \Diff^+(M) \mid \varphi^*\mu_\lambda = \mu_\lambda\}.
$$
\end{defn}

The following apparently weaker-looking
characterization of a strict contact pair was noticed in \cite{do-oh:reduction}. 
This new characterization of strict contact pairs
is one of the key ingredients which enables us to prove the generic scarcity of strict contact
diffeomorphisms. 

\begin{lem}[Lemma \ref{lem:do-oh-intersection}] \label{lem:do-oh-intro}  We have
$$
\Cont^{\text{\rm st}}(M,\lambda) = \Cont (M,\lambda) \cap \Diff(M,\mu_\lambda).
$$
\end{lem}

Motivated by this characterization and the invariance of the strict contact pair equation
$\psi^*\lambda = \lambda$, we first consider the subset
consisting of $\lambda$ with $\vol(\lambda) = 1$ where $\vol(\lambda): = \int_M d\mu_\lambda$.
We restrict our attention to the subset
\be\label{eq:C1M}
\mathfrak{C}_1(M): = \{\lambda \in \mathfrak{C}^+(M) \mid \vol(\lambda) = 1\}.
\ee
We would like to emphasize that \emph{this subset is invariant under the pull-backs by
the orientation-preserving diffeomorphisms, and forms a slice of the aforementioned conformal rescaling.}
In particular, we have a homeomorphism $\mathfrak{C}_1(M) \cong \mathfrak{C}_1(M)/\sim$.

We introduce fiber bundles of groups
\bea\label{eq:frak-Diff}
\mathfrak{Diff}_1(M) & = & \bigcup_{\lambda \in \mathfrak{C}_1(M)} \{\lambda\} \times \Diff(M) \nonumber \\
\mathfrak{Diff}_1(M)_0 & = & \bigcup_{\lambda \in \mathfrak{C}_1(M)} \{\lambda\} \times \Diff (M,\mu_\lambda)_0
\eea
respectively as a trivial fibration over the same base $\mathfrak{C}_1(M)$, 
where $\Diff (M,\mu_\lambda)_0$ denotes
the identity component of $\Diff (M,\mu_\lambda)_0$. \emph{Since we will always work with 
the latter throughout the paper,
we will just denote it as $\Diff (M,\mu_\lambda)$ dropping $0$ from notations.}
We denote by 
$
\mathfrak{Cont}^{\text{\rm st}}(M) \subset \mathfrak{Cont}(M)
$
the set of strict contact pairs, where we put
\be\label{eq:strict-contact-pair-set}
\mathfrak{Cont}_1^{\text{\rm st}}(M) = \bigcup_{\lambda \in \mathfrak{C}_1(M)} \{\lambda\} \times \Cont^{\text{\rm st}}(M,\lambda).
\ee

When $\lambda$ is a contact form of a given contact structure $\xi = \ker \lambda$, 
it uniquely defines the Reeb vector field $R_\lambda$
by the defining condition
\be\label{eq:Rlambda}
R_\lambda \intprod d\lambda = 0, \, \lambda(R_\lambda) = 1
\ee
and admits a decomposition
\bea
TM & = & \xi \oplus \R\langle R_\lambda \rangle \label{eq:TM-decomposition}\\
T^*M & = & (\R\langle R_\lambda \rangle)^\perp \oplus (\xi)^\perp \label{eq:T*M-decomposition}
\eea
where $(\cdot)^\perp$ denotes the annihilator of $(\cdot)$. Denote by $\pi: TM \to \xi$
(resp. $\pi: T^*M \to (\R\langle R_\lambda \rangle)^\perp$) the projection with respect to
these splittings, respectively.

We have the canonical  subset of $\mathfrak{Cont}_0^{\text{\rm st}}(M,\mu)$
$$
\mathfrak{Reeb}_1(M): = \bigcup_{\lambda \in \mathfrak{C}_1(M)} \{\lambda\} \times 
\text{\rm Reeb}(M,\lambda).
$$
We introduce the fiber bundle
$$
\Omega^1_{\mathfrak C_1}(M): = \bigcup_{\lambda \in \mathfrak{C}_1(M)}
\{\lambda\} \times \Omega^1_\pi(M,\lambda), 
$$
where we define the subset consisting of one-forms 
\be\label{eq:Omega1piM}
\Omega^1_\pi(M,\lambda): = \{\alpha^\pi \in \Omega^1(M) \mid \alpha^\pi = Z^\pi \intprod d\lambda, \, Z^\pi \in \Gamma(\ker \lambda)\}.
\ee
 It is easy to check that this subset is nothing but the set of sections of  the subbundle
$$
(\R\langle R_\lambda \rangle)^\perp \subset T^*M.
$$
Then we consider the map
$$
\Phi :\mathfrak{Diff}_1(M,\mu) \to  \Omega^1_{\mathfrak C_1}(M), 
$$
defined by
\be\label{eq:Phi-intro}
\Phi(\lambda,\psi) = (\psi^*\lambda)^\pi
\ee
with the pairs
$$
(\lambda,\psi) \in \mathfrak{Diff}_1(M,\mu) = \bigcup_{\lambda \in  \mathfrak{C}_1(M)} \{\lambda\} \times \Diff(M,\mu_\lambda).
$$
We have the following characterization of strict contact pairs
$$
\mathfrak{Cont}_1^{\text{\rm st}}(M) = \Phi^{-1}(0)
$$
as the zero set of the map $\Phi$. The full set of strict-contact pairs is given by the union
$$
\mathfrak{Cont}^{\text{\rm st}}(M,\mu) = \bigcup_{C > 0} \mathfrak{Cont}_C^{\text{\rm st}}(M,\mu)
$$
where we define
$$
\mathfrak{Cont}_C^{\text{\rm st}}(M): = \bigcup_{\lambda \in C\cdot \mathfrak{C}_1(M)} \{\lambda\} \times \Cont^{\text{\rm st}}(M,\lambda).
$$
An apparent symmetry of the equation arises from  the diagonal action of $\Diff(M)$ is given by
\be\label{eq:diff-action}
(\phi, (\lambda,\psi)) \mapsto ((\phi^{-1})^*\lambda,\phi \circ \psi).
\ee
More specifically the map $\Phi: \mathfrak{Cont} (M,\mu) \to \Omega^1(M)$ satisfies 
$$
\Phi ((\phi^{-1})^*\lambda,\phi \circ \psi) = 0 \Longleftrightarrow \Phi(\lambda,\psi) = 0
$$
for all diffeomorphism $\phi$, and in particular $\mathfrak{Cont}_1^{\text{\rm st}}(M,\mu)$ is invariant 
under this diagonal action of $\Diff(M)$. It makes  the equation of strict contact pairs carry
an \emph{infinite dimensional symmetry}.

The following parametric submersive property is the first step towards the proof of
main theorem. We would like to emphasize that this smoothness will follow
from the submersion property of the map $\Phi$ whose domain is the space
$$
\mathfrak{Diff}_1(M,\mu)
$$
fibered over the \emph{full} space $\mathfrak{C}_1(M) \cong \mathfrak{C}(M)/\sim$.

\begin{thm}\label{thm:Phisubmersion2-intro}
The subset
 $$
\mathfrak{Cont}_1^{\text{\rm st}}(M) \subset\mathfrak{Diff}_1(M,\mu) 
 $$
consisting of  strict contact pairs is a (Frechet) smooth submanifold of 
$$
\mathfrak{Diff}_1(M,\mu) = \bigcup_{\lambda \in \mathfrak{C}_1(M)} 
\{\lambda\} \times \Diff(M).
$$
\end{thm}

Following the standard procedure of applying the relevant Fredholm analysis, we then consider the projection map 
$\Pi: \mathfrak{Cont}_1^{\text{\rm st}}(M) \to \mathfrak{C}_1(M)$ given by
$$
\Pi(\psi, \lambda) = \lambda
$$
which is the restriction of the first projection map
$$
\pi_1: \mathfrak{Cont}_1^{\text{\rm st}}(M,\mu) \to \mathfrak{C}_1(M)
$$
to $\mathfrak{C}^{\text{\rm st}}(M)$. We highlight that \emph{the codomain of $\Pi$ is 
the full space $\mathfrak{C}_1(M)$.} We would like to study whether $\Pi$ is a
submersion or not. We denote by $\Pi_\lambda$ the restriction of $\Pi$ to the fiber $\Pi^{-1}(\lambda)$.

It turns out that there is a serious obstruction for $\lambda$ to satisfy this submersive property of this projection 
map $\Pi_\lambda$ which is summarized in the following proposition.

\begin{prop}[Proposition \ref{prop:dPi-cokernel}]\label{prop:cokernel-intro}
Let $J$ be a CR-almost complex structure and equip $M$ with the triad metric $g= g_\lambda$ of the contact
triad associated with the triad $(M,\lambda, J)$. Then we have the kernel of the
$L^2$-adjoint $(D\Pi_\lambda)^\dagger$ of $D\Pi_\lambda$ given by
$$
\coker D\Pi_\lambda \cong \ker (D\Pi_\lambda)^\dagger (\lambda,\psi) = \{f\, \lambda \in C^\infty_{(0;\lambda)}(M) \mid R_\lambda[f] = 0\}.
$$
\end{prop}

Here enters Definition \ref{defn:nonprojectible} and Theorem \ref{thm:nonprojectible-intro}
borrowed from \cite{oh:nonprojectible}, which will imply that a contact form $\lambda$ is a regular value 
i.e.,
$$
\coker D\Pi_\lambda = \{\text{\rm constant functions}\},
$$
\emph{provided $\lambda$ is non-projectible}. This will be how non-projectible contact forms become 
relevant to the main theorem of the present paper.

\subsection{Floer's $C^\varepsilon$ Banach norms in $C^\infty$ topology}

For the Fredholm analysis of the linearized operator $D\Phi(\lambda,\psi)$,
we need to suitably complete the linear differential operator
$D\Phi(\lambda, \psi)$ to a bounded linear operator with respect to a suitable Banach
spaces, so that the relevant  projection 
$$
\Pi_\lambda : \Phi_\lambda^{-1}(0,0) \to \mathfrak{C}_1(M); (\alpha,X) \mapsto \alpha
$$
becomes  a (nonlinear) Fredholm map on some Banach completions of its domain and
codomain.  

However, unlike other nonlinear elliptic problems, for example that of  pseudoholomorphic curves
in symplectic geometry,  the vertical linearization of $\Phi$,
 $$
 D\Phi_\lambda(\psi), \quad \Phi_\lambda(\psi): = \Phi(\lambda,\psi)
 $$
 is not of elliptic-type, and the kernel can be infinite dimensional in general. 
 
This, together with the fact that 
\emph{the standard $C^r$ topology of $\Cont(M,\lambda)$ as a subset of $\Diff^r(M)$ is not
known to be locally contractible}, is why it is essential for us to take Floer's $C^\varepsilon$ norm on 
the function spaces such as $C^\infty(M,\R)$ and directly work with $C^\infty(M,\R)$
in the perturbation problem.

\begin{rem}
We refer to the introduction of \cite{oh:simplicity} for the remark on some choice of modified $C^k$
topology of $\Cont(M,\lambda)$ that is locally contractible, which was
first constructed by Lychagin \cite{lychagin}. However this topology behaves
differently in its regularity depending on the $\xi$-direction or on the Reeb direction.
See \cite{tsuboi:simplicity}, \cite{oh:simplicity} how the $C^k$-estimates are handled in the proof of
simplicity of the contactomorphism group of finite regularity. From the way how the estimates
are carried out, it is not clear to us that one could  work with this modified
$C^k$ topology for the present problem more easily than directly working with $C^\infty$ topology.
\end{rem}
It is quite interesting to see that in our study
the way how Floer's $C^\varepsilon$-norms  is applied is quite different from that of 
pseudoholomorphic curves in that the $C^\varepsilon$-norm is used from the beginning of 
the main differential operator $R_\lambda$ as well as for the parameter spaces.
In the study of the linearized operator of the parameterized nonlinear Cauchy-Riemann
operator $\delbar$,
$$
D\delbar(v,J) = D_1\delbar(v,J) + D_2\delbar (v,J), \quad \delbar(v,J): = \delbar_J v
$$
for which
the main part $D_1\delbar(v,J) = D\delbar_J(v)$ is a first-order linear elliptic 
operator so that we commonly use $C^k$-norms via the coercive estimates.
Floer's $C^\varepsilon$-norm is usually used only for the study of the parameter space 
related to the second partial derivative $D_2\delbar (v,J)$. (See \cite{floer:unregularized} and
\cite{wendl:lecture,wendl:super-rigidity}.)

\begin{exm} \label{exm:S2n+1} Consider the standard contact form of $S^{2n+1} \subset \C^{n+1}$ which is 
a contact manifold as the prequantization of $\C P^n$. It carries an infinite dimensional
set of strict contactomorphisms generated by the Hamiltonians of the form
$$
H = \pi^*h, \quad h \in C^\infty(\C P^n,\R)
$$
Theorem \ref{thm:nostrict-intro} imply that 
the conformal class  $[\alpha_{\text{\rm st}}]$ of standard contact form
cannot be a regular value of  the map 
$$
\Pi_\Phi(\lambda,\psi)= \lambda
$$
defined on $\Phi^{-1}(0)$ where $\Phi(\lambda,\psi): = (\psi^*\lambda)^\pi$.
Almost all of these contactomorphisms, other than the Reeb flows, will disappear after
a $C^\infty$-small perturbation of $\lambda_{\text{\rm st}}$.
\end{exm}

\medskip
\noindent{\bf Acknowledgement:} We thank V. Ginzburg for letting us know about the problem.
We also thank Institute for Advanced Study for enabling this collaboration, and financial support. 
The first-named author also thanks Hyuk-Seok Do for the collaboration of \cite{do-oh:reduction}  
which provides  an important insight on the characterization of strict contactomorphisms 
employed in the present paper.

\bigskip

\noindent{\bf Convention and Notations:}

\medskip

\begin{itemize}
\item {(Contact Hamiltonian)} We define the contact Hamiltonian of a contact vector field $X$ to be
$- \lambda(X) =: H$.
\item
For given time-dependent function $H = H(t,x)$, we denote by $X_H$ the associated contact Hamiltonian vector field whose associated Hamiltonian $- \lambda(X_t)$ is given by $H = H(t,x)$, and its flow by $\psi_H^t$.
\item {(Reeb vector field)} We denote by $R_\lambda$ the Reeb vector field associated to $\lambda$
and its flow by $\phi_{R_\lambda}^t$. It is the same as $X_H$ with $H \equiv -1$.
\item $g_{(\psi;\lambda)}$: the conformal exponent defined by 
$\psi^*\lambda = e^{g_{(\psi;\lambda)}} \lambda$
for the contact pair $(\lambda,\psi)$.
\end{itemize}

\section{Basic contact Hamiltonian geometry and dynamics}
\label{sec:contact-Hamiltonian-geometry}

We set up our conventions for the definitions of Hamiltonian vector fields both in
symplectic and in contact geometry so that they are compatible in some natural sense.
We briefly summarize basic calculus of contact Hamiltonian dynamics to set up our conventions
on their definitions and signs following \cite{oh:contacton-Legendrian-bdy}. 

Now let   $(M,\xi)$ be a contact manifold of dimension $m = 2n+1$, which is coorientable.
We denote by $\Cont_+(M,\xi)$ the set of orientation preserving contactomorphisms
and by $\Cont_0(M,\xi)$ its identity component.  
Equip $M$ with a contact form $\lambda$ with $\ker \lambda = \xi$. 

\subsection{Conformal exponents of contactomorphisms}

For a given contact form $\lambda$, a coorientation preserving diffeomorphism $\psi$ of $(M,\xi)$
is contact if and only if it satisfies
$$
\psi^*\lambda = e^g \lambda
$$
for some smooth function $g: M \to \R$, which depends on the choice of contact form $\lambda$ of $\xi$.
Unless said otherwise, we will always assume that $\psi$ is coorientation preserving without mentioning
from now on.

\begin{defn} For given coorientation preserving contact diffeomorphism $\psi$ of $(M,\xi)$ we call
the function $g$ appearing in $\psi^*\lambda = e^g \lambda$ the \emph{conformal exponent}
for $\psi$ and denote it by $g=g_{(\psi;\lambda)}$.
\end{defn}

The following lemma is a straightforward consequence
of the identity $(\phi\psi)^*\lambda = \psi^*\phi^*\lambda$.

\begin{lem}\label{lem:coboundary} Let $\lambda$ be given and denote by $g_\psi$ the function $g$
appearing above associated to $\psi$. Then
\begin{enumerate}
\item $g_{\phi\psi} = g_\phi\circ \psi + g_\psi$ for any $\phi, \, \psi \in \Cont(M,\xi)$,
\item $g_{\psi^{-1}} = - g_\psi \circ \psi^{-1}$ for all $\psi \in \Cont(M,\xi)$.
\end{enumerate}
\end{lem}

An immediate corollary of the definition is the following integral identity.

\begin{cor}\label{cor:normalization}
Assume $M$ is compact and $\xi$ is coorientable. Let $g_{(\psi;\lambda)}$ be conformal exponent as above. Then we have
\be\label{eq:normalization}
\frac{1}{\vol_\lambda(M)} \int_M e^{g_{(\psi;\lambda)}}\, \lambda \wedge (d\lambda)^{n} = 1
\ee
where $\vol_\lambda(M): =  \int_M \, \lambda \wedge (d\lambda)^{n}$. In particular 
neither of $U^\pm(\lambda,\psi)$ is empty unless $g_{(\psi;\lambda)} = 0$, i.e., $\psi$ is 
$\lambda$-strict, where we define
$$
U^\pm(\lambda,\psi): = \{x \in M \mid \pm g_{(\psi;\lambda)}(x) > 0\}.
$$
\end{cor}
Recall the definition of the probability measure $\mu_\lambda$ 
associated to $\lambda$ mentioned in the introduction.
We will call $\mu_\lambda$ the \emph{$\lambda$ probability measure}, 
and call it a \emph{contact probability measure}
in general when we do not specify $\lambda$.

\subsection{Hamiltonian calculus of contact vector fields}

In this subsection, we summarize basic Hamiltonian calculi which as a whole will play 
an important role throughout the proofs of the main results of the present paper.
A systematic summary with the same sign convention is given in 
\cite[Section 2]{oh:contacton-Legendrian-bdy}. (See also \cite{dMV} for
a similar exposition with the same sign convention.)

\begin{defn} A vector field $X$ on $(M,\xi)$ is called \emph{contact} if $[X,\Gamma(\xi)] \subset \Gamma(\xi)$
where $\Gamma(\xi)$ is the space of sections of the vector bundle $\xi \to M$.
We denote by $\mathfrak X(M,\xi)$ the set of contact vector fields.
\end{defn}

Then we have the following unique decompositions of
a  vector field $X$ and one-forms $\alpha$ in terms of \eqref{eq:TM-decomposition} and
\eqref{eq:T*M-decomposition} respectively.  They will play an important role in 
 various calculations entering in the calculus of contact
 Hamiltonian geometry and dynamics.
 
 \begin{lem}\label{lem:decomposition}
  Let $X$ be a vector field and $\alpha$ a one-form on $(M,\lambda)$. Then
 we have the unique decompositions 
\be\label{eq:X-decompose}
X = X^\pi + \lambda(X) R_\lambda,
\ee
and 
\be\label{eq:alpha-oneform}
\alpha = \alpha^\pi + \alpha(R_\lambda) \lambda
\ee
respectively. Here we use the unique representation
$\alpha^\pi = Z^\pi \intprod d\lambda$ for $\Z^\pi \in \Gamma(\xi)$.
\end{lem}

Next, the condition $[X,\Gamma(\xi)] \subset \Gamma(\xi)$ is
 equivalent to the condition that
there exists a smooth function $f: M \to \R$ such that
$$
\CL_X \lambda = f \lambda.
$$
\begin{defn} Let $\lambda$ be a contact form of $(M,\xi)$.
The associated function $H$ defined by
\be\label{eq:contact-Hamiltonian}
H = - \lambda(X)
\ee
is called the \emph{$\lambda$-contact Hamiltonian} of $X$. We also call $X$ the
\emph{$\lambda$-contact Hamiltonian vector field} associated to $H$.
\end{defn}
We alert readers that under our sign convention, the $\lambda$-Hamiltonian $H$ of the Reeb vector field $R_\lambda$
as a contact vector field becomes the constant function $H = -1$.

Conversely, any smooth function $H$ associates a unique contact vector field
that satisfies the relationship spelled out in the above definition. We highlight
the fact that unlike the symplectic case, this correspondence is one-one with
no ambiguity of addition by constant. 
We denote by $R_\lambda$ the Reeb vector field of $\lambda$.

We next apply the aforementioned decomposition of vector fields and one-forms to
the contact vector fields and to the exact one-forms respectively.

\begin{prop} Equip $(M,\xi)$ with a contact form $\lambda$.
Let $X$ be a contact vector field and consider the decomposition
$X = X^\pi + \lambda(X) R_\lambda$.  If we set $H = -\lambda(X)$, then $dH$ 
satisfies the equation
\be\label{eq:dH}
dH = X^\pi  \intprod d\lambda  + R_\lambda[H]\, \lambda.
\ee
\end{prop}
In fact, for any smooth function $H$, the associated contact vector field $X$
 is uniquely determined by the equation
\be\label{eq:XH-defining-eq}
X \intprod \lambda = -H, \quad X \intprod d\lambda = dH - R_\lambda[H] \lambda.
\ee
We denote this $X$ by $X_H$ and call the \emph{contact Hamiltonian vector field}
of $H$.
\begin{lem}\label{cor:LXtlambda} Let $X$ be a contact vector field with $\CL_X \lambda = f \lambda$,
and let $H$ be the associated contact Hamiltonian. Then $f = -R_\lambda[H]$.
\end{lem}

For each given contact form $\lambda$,
any smooth function $H$ gives rise to a contact vector field, denoted as $X_H$.
We will also write
$$
X_H = X_{(H;\lambda)}
$$
following the notation from \cite{oh:nonprojectible},
when we need to emphasize the $\lambda$-dependence of the expression $X_H$.
 
\section{Volume forms, metric divergence and Lie derivative}
\label{sec:volume}

For each given contact pair $(\lambda,\psi)$,  we now introduce metrics on $M$ depending on
the contact triads $(M,\lambda,J)$ associate to CR-almost complex structures $J$.

\begin{lem}\label{lem:dvol=dmulambda} Consider the contact triad $(M,\lambda,J)$ and
its associated triad metric $g$ and the volume form $d\vol$.  Then we have
$$
d\vol = d\mu_\lambda.
$$
\end{lem}
\begin{proof}
We first recall the well-known definition of the (metric) divergence $\nabla\cdot X$ 
of a general vector field that
we have
\be\label{eq:CLX=divergenceX}
\CL_{X_t} d\vol = ( \nabla\cdot X_t ) d\vol.
\ee
To prove the coincidence of the two volume forms, we consider the 
Darboux frame
$$
\{E_1, \cdots, E_n, F_1, \cdots, F_n, R_\lambda\}
$$
such that $E_i, \, F_j \in \xi$ and $F_j = J E_j$ and so $d\lambda(E_i,F_j) = \delta_{ij}$.
 Obviously we have
$$
d\mu_\lambda(E_1, \cdots, E_n, F_1, \cdots, F_n, R_\lambda) = \lambda(R_\lambda)
\cdots d\lambda(E_1,F_1)\cdots d\lambda(E_n,F_n) = 1.
$$
On the other hand, the above frame is also an orthonormal frame of 
the triad metric $g$ and hence $d\vol(E_1,\cdots, E_n, F_1, \cdots, F_n, R_\lambda) = 1$.
This finishes the proof.
\end{proof}

\begin{defn} We define the subset 
\be\label{eq:vol-1}
\mathfrak{C}_1(M): = \{\lambda \in \mathfrak{C}_1(M) \mid \vol(\lambda) = 1\}
\ee
We call any such contact form \emph{$\vol$-normalized}.
\end{defn}

The following identification of the tangent space of the subset $\mathfrak{C}_1(M)$ given in 
\eqref{eq:C1M} will play an important role in our proof.

\begin{prop}\label{prop:isovolume-alpha}\label{prop:C1-tangent}
Let $\lambda$ be given and consider its associated contact distribution
$\xi = \ker \lambda$.  Assume the first variation $\alpha = \delta \lambda$ is tangent to
$\mathfrak{C}_1(M)$. Then we have
$$
T_\lambda\mathfrak{C}_1(M) = \left\{\alpha  \in \Omega^1(M) \, \Big|\, \int_M \alpha(R_\lambda) \, d\mu_\lambda = 0\right\}
$$
\end{prop}
\begin{proof} We compute
$$
\delta_\alpha \mu_\lambda = \alpha \wedge (d\lambda)^n + n \lambda \wedge d\alpha \wedge (d\lambda)^{n-1}.
$$
We decompose $\alpha = Z^\pi \intprod d\lambda + h\, \lambda$ for some $Z^\pi \in \Gamma(\xi)$ 
as in Lemma \ref{lem:decomposition} with $h = \alpha(R_\lambda)$.
For any contact form  $\lambda$ and the associated decomposition $\alpha = \alpha^\pi + h_\alpha \lambda$,
we derive
$$
\alpha \wedge (d\lambda)^n + n \lambda \wedge d\alpha \wedge (d\lambda)^{n-1}
= (n+1) h\, \lambda \wedge (d\lambda)^n.
$$
Then, since $\alpha \in T_\lambda \mathfrak{C}_1(M)$, we have
$$
0 =  \delta_\alpha \left(\int_M \mu_\lambda \right) 
= \int_M \delta_\alpha(\mu_\lambda) = \int_M (n+1) h\, \lambda \wedge (d\lambda)^n
$$
and hence $\int_M h \, d\mu_\lambda= 0$. 
\end{proof}

For any contact pair $(\lambda,\psi)$, it defines the conformal exponent function $g_{(\lambda,\psi)}$
defined by $\psi^*\lambda = e^{g_{(\lambda,\psi)}}\lambda$.  By definition,
\emph{if $M$ is compact}, the function has the natural constraint
$$
\int_M e^{g_{(\psi,\lambda)}} \, d\mu_\lambda = \int_Md\mu_\lambda =: \vol(\lambda)
$$
In other words, the factor $ e^{g_{(\psi,\lambda)}}$ is nothing but
the Radon-Nikodym derivative
$$
\frac{ d\psi^*\mu_\lambda}{d\mu_\lambda}
$$
which is defined a general pair of measures $(\mu', \mu)$. This provides a canonical extension of
the \emph{relative} conformal factor associated to a pair of contact forms $(\psi^*\lambda, \lambda)$
to that of the pair $(\psi^*\mu_\lambda,\mu_\lambda)$ of volume forms for
general orientation preserving diffeomorphism,  which is not necessarily a contactomorphism.
More precisely the $(n+1)$-th power of the former to the latter.

For a further discussion, we start with the following crucial observation made by
Do and the first-named author in \cite{do-oh:reduction}.

\begin{lem} [Remark 5.5 \cite{do-oh:reduction}]\label{lem:do-oh-intersection} We have
$$
\Diff(M) \cap \Cont(M,\lambda) = 
\Cont^{\text{\rm st}}(M,\lambda).
$$
\end{lem}
\begin{proof} For any contact pair $(\psi, \lambda)$ of $\xi=\ker \lambda$, we have
$$
\psi^*\lambda = f \lambda, \quad f > 0.
$$
Therefore we have
$$
\psi^*\mu_\lambda = (f \lambda) \wedge (d(f\lambda))^n = f^{n+1} \lambda \wedge (d\lambda)^n.
$$
Therefore if $\psi^*\mu_\lambda = \mu_\lambda$, then $f^{n+1} \equiv 1$ and hence $f \equiv 1$
because $f > 0$. This proves $\psi^*\lambda = \lambda$.
\end{proof}

This motivates us to later consider the map 
$$
\Phi(\lambda,\psi) = (\psi^*\lambda)^\pi
$$
whose zero set is precisely the set of $\lambda$-contactomorphisms

\section{Floer's $C^\varepsilon$ norm and the off-shell framework}
\label{sec:offshell-framework}

As mentioned before, the operator $R_\lambda$ is not coercive and so the standard 
 $C^r$-completion of $\mathfrak{C}_{1}(M,\R)$ and
the completion of the map $\Psi^\lambda$ thereto is not suitable for the Fredholm analysis.
For this purpose, we utilize  the set of smooth functions equipped with
Floer's $C^\varepsilon$ Banach norm whose explanation is now in order.
A nice account of this $C^\varepsilon$
topology is given by Wendl \cite[Section 5.4]{wendl:super-rigidity}, \cite[Appendix B]{wendl:lecture}.

Following \cite{floer:unregularized}, we consider a sequence of positive numbers 
$\varepsilon = \{ \varepsilon_k\}_{k = 0}^\infty$ and define the space denoted by
$$
C^\varepsilon(M,\R)
: = \left\{ f \in C^\infty(M,\R) \,\Big| \, \sum_{k=0}^\infty \varepsilon_k |D^k f| < \infty \right\}
$$
where we take the norm with respect a contact triad connection. Then Floer's \cite[Lemma 5.1]{floer:unregularized} proves that if $\epsilon$ decays
sufficiently fast, the space is a Banach space.  We will also consider its $C^1$-version
thereof:
$$
C^{(1,\varepsilon)}(M,\R) : = C^\varepsilon(M,\R) \cap \left\{ f \in C^\infty(M,\R) \,\Big| \, \|Df\|_\varepsilon 
< \infty \right\}.
$$
For the later purpose, we also need to make the definition $D^k f$  and the norm $\|\cdot \|_\varepsilon$
more precise. This is because the norm $\|\cdot \|_\varepsilon$ is not intrinsic \emph{on a manifold}
but depend on many choices such as the choice of atlases or of the connections $\nabla$.

One can easily define a similar $C^\varepsilon$-norm on the space $\Gamma(E)$ of sections of
a vector bundle $E \to M$. This being said, we focus on the case of $C^\infty(M,\R)$ for the simplicity
of exposition.

\subsection{Chain rule for the derivatives of composition maps}
\label{subsec:estimates-summary}

In this subsection, we collect basic facts on the norms and 
estimates of functions, especially of the composition of the type
$$
f\circ \psi
$$
where $f: M \to \R$ is a real-valued function and $\psi:M \to M$ is a 
diffeomorphism. We will especially focus on the case of
 the Euclidean spaces since other cases will be based on it. Such estimates 
 have been extensively exploited by Mather \cite{mather,mather2} and 
 Epstein \cite{epstein:commutators}
in relation to their study on the simplicity of the diffeomorphism groups. Our case
is much simpler than those used in them.

Let $f: U \to \R$ be a $C^r$-function, where $U$ is an open subset of $\R^n$. We define
$$
\|f\|_r: = \sup_{x \in U} \|D^rf(x)\|.
$$
If ${\bf f} = (f_1, \ldots, f_k)$ is a $k$-tuple of $C^r$-functions, we write
$\|{\bf f}\|_r = \sup_{1 \leq i \leq k} \|f_i\|_r$.

Now let $g: M \to M$ be a diffeomorphism. Then we have the formula
$$
D(f \circ g) = (Df \circ g)\cdot (Dg)
$$
where the right hand side is a composition of two linear maps, with right 
matrix multiplication by an  $n \times n$ matrix  $Dg$ multiplied to a row vector
$(Df \circ g) = df \circ g$. For the higher derivatives, we have
\bea\label{eq:higher-derivative-product}
D^r (f\circ g) & = & (D^r f)(Dg \times \cdots \times Dg) + (Df \circ g)(D^rg) \nonumber\\
&{}&  + \sum C(i;j_1, \ldots, j_i) (D^if \circ g)(D^{j_1}g \times \cdots \times D^{j_i} g)
\eea
where $C(i;j_1, \ldots, j_i)$ is an integer which is independent of $f$, $g$ and 
even of dimensions of their domains and codomains for
$$
1 < i < r, \quad j_1 + \ldots + j_i = r, \quad j_s \geq 1.
$$
We recall that $(D^if \circ g)$ is a \emph{multilinear map} of $i$ arguments.
\emph{This implies that we have $j_s \geq 2$ at least one $s$}. For the simplicity of notation, we write
$$
D^Jg = (D^{j_1} \times \cdots \times D^{j_i}) g, \quad J = (j_1, \ldots, j_i).
$$
Then we can concisely write \eqref{eq:higher-derivative-product} into
\be\label{eq:Drfg}
D^r(f \circ g) = (D^r f)(Dg \times\cdots \times Dg) + (Df \circ g)(D^rg)
+ \sum C(i;J) (D^if \circ g)(D^J g).
\ee

\subsection{Definition of Floer's $C^\varepsilon$ norms}

For the purpose of the precise definition of Floer's norm,
we will amalgamate the two approaches by writing down the norms in the covariant way but taking
the coordinate approach  for actual calculations using Darboux charts to avoid quite complicated higher-order
tensorial calculations. We cover $M$ with a finite number of Darboux chart with canonical coordinates 
$$
(q_1, \cdots, q_n, p_1, \cdots, p_n ,z)
$$
with $\lambda = dz - \sum_i p_i dq_i$ so that $R_\lambda = \frac{\del}{\del z}$.
Note that in the Darboux chart, $R_\lambda = \frac{\del}{\del z}$
commutes with all basic operators 
$$
\mathfrak B = \left\{\frac{\del}{\del z}, \frac{\del}{\del q_i}, \frac{\del}{\del p_i} \right\}.
$$
\emph{We will also need to choose $\epsilon_k \to 0$ sufficiently fast as $k \to 0$ whose 
speed of decay will depend on the atlas $\mathfrak{A}$.}

\begin{defn} Let $\mathfrak{A} = \{(U_\alpha,\varphi_\alpha)\}_{\alpha=1}^N$ be a
finite atlas each element of which is a Darboux chart with its image
$$
\varphi_\alpha(U_\alpha) = (-1,1)^{2n+1} \subset \R^{2n+1}.
$$
Let $\epsilon = \{ \epsilon_i\}_{i=0}^\infty$ be such that $\epsilon_k \to 0$ sufficiently fast.
For given compactly supported $g \in C^\infty(M,\R)$, we define the norm 
$$
\|f\|_{\epsilon, \mathfrak{A}} = \sum_{\alpha=1}^N \sum_{k=0}^\infty
 \epsilon_i \|D^k (f\circ \varphi_\alpha^{-1})\|_{C^0( (-1,1)^{2n+1})}
 $$
 and call it the $C_{\varepsilon;\mathfrak A}$-norm.
 We denote by $C_{\varepsilon;\mathfrak A}(M,\R)$ the set of function $f$ with bounded 
 $C_{\varepsilon;\mathfrak A}$-norm.
\end{defn}
Here  $\|D^k g|_{C^0((-1,1)^{2n+1})}$ is the standard norm on $(-1,1)^{2n+1}$ given by
 \be\label{eq:|Dkg|}
 \|D^k g\|_{C^0( (-1,1)^{2n+1})}
 : =  \max_{I \subset \mathfrak{B}} \| D^I g(x)\|_{C^0( (-1,1)^{2n+1})}.
 \ee
More explicitly, we have
$$
\|D^k g\|_{C^0( (-1,1)^{2n+1})} = \max_{I\subset \mathfrak{B}} 
\left\|\frac{\del^{|I|} g}{\del x_{i_1} \cdots \del x_{i_k}} \right\|_{C^0( (-1,1)^{2n+1})}
$$
where $I = (i_1,\cdots, i_k)$ and $|I| = \sum_{i=1} i_k$.
\emph{Since $\mathfrak A$ will not vary but be fixed, we will drop $\mathfrak{A}$ from 
the notations in the rest of the paper, unless there is a danger of ambiguity or confusion.}

\subsection{Definition of $C^{(1,\varepsilon)}$-norms}

We now define the Banach space
$$
C^{(1,\varepsilon)}(M,\R): = \{ f \in C^\varepsilon(M,\R) \mid Df \in C^\varepsilon(M,\R)\}
$$ 
which is equipped with the $C^{(1,\varepsilon)}$-norm by
\be\label{eq:C1epsilon-norm}
\|f\|_{(1,\varepsilon)}: = \|f\|_0 + \|Df\|_\varepsilon.
\ee
By definition, then we have the inequality
$$
\|Df\|_\varepsilon \leq \|f\|_{(1,\varepsilon)}
$$
and hence the map $D: C^{(1,\varepsilon)}(M,\R) \to C^{\varepsilon}(M,\R)$ is a
continuous map.

We prove the following proposition which will be a crucial ingredient in
our proof of the main theorem.

\begin{prop}\label{prop:D-continuity} Let $\lambda$ be a contact form and let $\mathfrak A$
be an associated atlas of Darboux charts.
\begin{enumerate}
\item 
Then there exists a sequence $\varepsilon = \{\varepsilon_k\}$ such that the space 
$C^{(1,\varepsilon)}(M,\R) \subset C^\varepsilon(M,\R)$ is non-empty, and 
both $C^\varepsilon(M,\R)$ 
and $C^{\varepsilon,1}(M,\R)$ are dense in $C^\infty(M,\R)$ (or in $C^m(M,\R)$ for any  $m \geq 1$.
\item 
The operator
$$
D: C^{(1,\varepsilon)}(M,\R) \to C^\varepsilon(M,\R)
$$ 
is defined and continuous in the norm topology.
\end{enumerate}
\end{prop}
\begin{proof} Recalling that $C^\infty(M,\R)$ equipped with $C^\infty$ topology is separable, 
we consider a countable sequence $\{f_\nu\}_{\nu = 1}^\infty$ that is dense. We denote
$$
\delta^{(\nu)}_k: = \frac{\|Df_\nu\|_k}{\|f_\nu\|_k}
$$
for each $k = 0, \ldots, \to \infty$. We then select a sequence $\varepsilon^{(\nu)} = \{\varepsilon^{(\nu)}_k\}_{k=0}^\infty$
so that $\varepsilon^{(\nu)}_0 = 1$ for all $\nu$ and 
$$
\sup_{k} \left\{\varepsilon^{(\nu)}_k \delta^{(\nu)}_k\right\} \leq C < \infty
$$
for some constant $C> 0$ (e.g., $C =1$).
Then by the choice of $\varepsilon^{(\nu)}$ we obtain
\beastar
\|Df_\nu\|_{\varepsilon^{(\nu)}} & = & \sum_{k=0} \varepsilon^{(\nu)}_k \|D^k(Df_\nu)\| 
 = \sum_{k=1} {\varepsilon^{(\nu)}_{k-1}} \|D^k(f_\nu)\|\\
& = & \varepsilon^{(\nu)}_0 \|Df_{\nu}\| + \sum_{k=1} \varepsilon^{(\nu)}_{k-1} \|D^k(f_\nu)\| \\
& = & \varepsilon^{(\nu)}_0 \|Df_\nu\| + \sum_{k=1} \left(\frac{\varepsilon_{k-1}^{(\nu)}}
{{\varepsilon^{(\nu)}_k}}\right)  \varepsilon^{(\nu)}_k \|D^k(f_\nu)\|\\
& \leq & \varepsilon^{(\nu)}_0 \|D f_\nu\| + C \sum_{k=1} {\varepsilon^{(\nu)}_k} \|D^k(f_\nu)\| \\
& \leq & \max\{1,C\}  \sum_{k=0} {\varepsilon^{(\nu)}_k} \|D^k(f_\nu)\|  
= \max\{1,C\} \|f_{\nu}\|_{\varepsilon^{(\nu)}}.
\eeastar
We define a sequence $\varepsilon = \{\varepsilon_k\}$ by
\be\label{eq:epsilonk}
\varepsilon_k: = \min_{\nu\leq k} \{\epsilon^{(1)}, \cdots, \epsilon^{(k)}\}
\ee
which then satisfies $\frac{\varepsilon_k}{\varepsilon_k^{(\nu)}} \leq 1$ for all $\nu=1, \cdots \to \infty$.

Now we consider $C^\varepsilon(M,\R)$ for this choice of $\varepsilon$. Then we obtain
\be\label{eq:fnu-epsilon-norm}
\|Df_\nu\|_{\varepsilon^{(\nu)}} \leq  C \|f_\nu\|_{\varepsilon^{(\nu)}}
\ee
for all $\nu$ by writing $\max\{1,C\}$ to be $C$. 

Furthermore the inequality \eqref{eq:fnu-epsilon-norm} proves that both spaces
$$
C^{(1,\varepsilon)}(M,\R) \subset C^\varepsilon(M,\R) 
$$
contain $\{f_\nu\}$ and hence are nonempty and dense in $C^m(M,\R)$
for any integer $m \geq 1$. This finishes the proof.
\end{proof}

\subsection{Boundedness of the first-order linear differential operators}

For the study of the linearization operator 
$$
D\Phi(\lambda,\psi): T_{(\lambda,\psi)} \mathfrak{Diff}_1(M)\to C^\infty_{(0;\lambda)}(M,\R) 
\times \Omega^1_\pi(M,\lambda)
$$
which is a first-order differential operator, we state the following boundedness of
the linearized operator, the proof of which we postpone until Section \ref{sec:smoothness}.

\begin{lem}\label{lem:bounded-DPhi} The operator $D\Phi(\lambda,\psi)$ is bounded with respect to the
$C^{(1,\varepsilon)}$-norm of its domain and the $C^\varepsilon$-norm of the 
codomain of the operator.
\end{lem}

To illustrate how the two norms $C^{(1,\varepsilon)}$-norm and the $C^\varepsilon$-norm
will enter in the proof of this lemma, we give the proof of the following lemma.
\begin{prop}\label{prop:proper-continuity} Let $\mathfrak A$ be given and let $\epsilon$ be as above.
Then the operator
$$
R_\lambda: C^{(1,\varepsilon)}(M,\R) \to C^\varepsilon(M,\R)
$$ 
is continuous.
\end{prop}
\begin{proof} For the continuity, it is enough to prove boundedness of the operator. 
We recall that  $R_\lambda = \frac{\del}{\del z}$ on any Darboux chart.

\begin{lem} There exists a constant $C_1 = C_1(\epsilon, \mathfrak A,\nabla) > 0$ independent of 
$f$ and of the integer $r \geq 0$ such that 
$$
\|D^r(R_\lambda[f])\|_{C^0} \leq C_1 \|D^{r+1}f\|_{C^0}.
$$
\end{lem}
\begin{proof}
Note that in the Darboux chart, $R_\lambda = \frac{\del}{\del z}$ which
commutes with all basic operators $\{\frac{\del}{\del z}, \frac{\del}{\del q_i}, \frac{\del}{\del p_i}\}$.

Consider the coordinate expressions 
$$
f \circ \varphi_\alpha^{-1}, \quad u \circ \varphi_\alpha^{-1}
$$
for each $\alpha$. We compute
$$
D\left(\frac{\del}{\del z} (f \circ \varphi_\alpha^{-1})\right) 
= \frac{\del}{\del z}\left(D(f \circ \varphi_\alpha^{-1})\right)
$$
on $\varphi_\alpha(U_\alpha) \cap \varphi_\beta(U_\beta)$.
Then by the finiteness of the atlas,
there exists a constant $C_1 = C_1(\lambda, \mathfrak{A})$ such that  
$$
\left\|D^r\left(\frac{\del f}{\del z}\right)\right\|_{C^0} \leq C_1 \|D^{r+1}f\|_{C^0}
$$
because $\frac{\del}{\del z}$ is a first-order differential operator on compact manifold $\R^{2n+1}$.
The lemma follows from this local expression by summing them over and adjusting the constant $C_1$.
\end{proof}

This implies
\beastar
\|R_\lambda[f]\|_{\varepsilon} & = & \sum_{k=0}^\infty \epsilon_k \|D^k(R_\lambda[f])\|_{C^0}\\
&\leq & C_1 \sum_{k=0} \epsilon_k \|D^{k+1}f\|_{C^0} = C_1 \|Df\|_\varepsilon
\eeastar
This shows $R_\lambda$ has the operator norm $\|R_\lambda\| \leq C_1\|D\|$,
which finishes the proof of continuity.
\end{proof}

\subsection{Floer's $C^\varepsilon$ Banach bundles of perturbations}
\label{subsec:Floer-perturbation}

In this subsection, we give an explicit description of the perturbation data of 
contact pairs $(\lambda,\psi)$ for the map $\Phi$. We closely follow the exposition
given in \cite{floer:unregularized, wendl:superrigidity} which 
is applied to the transversality study of the moduli space of pseudoholomorphic 
curves under the perturbation of almost complex structures. 

For each given pair $(\lambda,\psi)$, we consider its perturbations of the form
of the form
$$
\left(\exp_\lambda(\alpha), \exp_\psi(X)\right)
$$
where $\alpha \in \Omega^1(M)$ a one-form and $X \in \mathfrak{X}(M,\mu_\lambda)$
a vector field on $M$, and $\exp$ \emph{stands for a fixed parametrization} of $C^\infty$-small
perturbations obtained by the elements 
$$
(\alpha, X) \in T_{(\lambda,\psi)} \mathfrak{Diff}_1(M,\mu).
$$
We also choose a CR-almost complex structure $J$ continuously over the choice of contact forms $\lambda$
and associate the continuous family of contact triads $(M,\lambda,J)$, and equip
$M$ with the triad metric associated thereto. Since we will fix one $J$ but not vary for each $\lambda$,
we will omit $J$ from the notations remembering $J = J(\lambda)$.

We then define the Banach bundles
\be\label{eq:C1epsilon-subset}
T\mathfrak{Diff}_1^{(1,\varepsilon)}(M,\mu): = \bigcup_{(\lambda,\psi) \in \mathfrak{Diff}_1(M)} 
\{\lambda\} \times \mathfrak{X}^{(1,\varepsilon)}_{(\lambda,\psi)}(M,\mu_\lambda)
\ee
over $\mathfrak C_1(M)$ where each fiber 
$$
\mathfrak{X}_{(\lambda,\psi)} ^{(1,\varepsilon)}(M) = T_{(\lambda,\psi)}\mathfrak{Diff}_1^{(1,\varepsilon)}(M)
$$
is the subset of $T_{(\lambda,\psi)} \mathfrak{Diff}_1(M,\mu)$ given by
\beastar
&{}& \mathfrak{X}_{(\lambda,\psi)} ^{(1,\varepsilon)}(M,\mu_\lambda) \\
& : = & \left \{(\alpha,X) \in T_\lambda \mathfrak{C}_1(M,\mu_\lambda)
\oplus T_\psi \Diff(M) \mid \|\alpha\|_{(1,\varepsilon)} < \infty,\,
\|X\|_{(1,\varepsilon)} < \infty \right\}
\eeastar
which we equip with $C^{(1,\varepsilon)}$-norm topology
induced by the triad metric associated to $(M,\lambda,J)$. We define the obvious analogue of this
for the $C^\varepsilon$. 

Furthermore we can define various counterparts of this construction for the various universal spaces
that appear in the present paper:
\beastar
\mathfrak{Diff}_1(M,\mu) & = & \bigcup_{\lambda \in \mathfrak{C}_1(M) }
\{\lambda\} \times \Diff(M,\mu_\lambda)\\
\mathfrak{Cont}_1(M) &: =& \bigcup_{\lambda \in \mathfrak{C}_1(M)} \{\lambda\} \times \Cont(M,\lambda)\\
\mathfrak{Cont}_1 (M)_0 &: =& \bigcup_{\lambda \in \mathfrak{C}_1(M)} \{\lambda\} \times \Cont(M,\lambda)_0
\eeastar
where $\Cont(M,\lambda)_0$ is the identity component of $\Cont(M,\lambda)$.
All these have
the natural structure of smooth Frechet manifolds: The smooth structure $\Diff(M,\mu_\lambda)$
is described by Ebin-Marsden \cite{ebin-marsden}. Since
$\Diff(M,\mu_\lambda)_0 \subset \Diff(M,\mu_\lambda)$ is an open subset thereof, it
carries the structure of smooth Frechet manifold as a smooth submanifold of $\Diff(M,\mu_\lambda)$.

On the other hand, the smooth topology of $\Cont(M,\lambda)$, which is locally contractible, is given by the contact Darboux-Weinstein
theorem. (See \cite{lychagin}, \cite{tsuboi:simplicity}, \cite{oh:simplicity}.)

\section{Reduction of the equation of strict contactomorphisms}

In this section and henceforth, 
we will study the equation $\psi^*\lambda = \lambda$ for the pair $(\lambda,\psi)$ which we
call the \emph{equation of strict contactomorphisms}. 

\begin{defn}[Strict contact pair] Let $\lambda$ be a contact form and $\psi \in \Diff(M)$.
We call a pair $(\lambda,\psi)$  a \emph{strict contact pair} if
it satisfies the equation $\psi^*\lambda = \lambda$.
\end{defn}

We consider the set \eqref{eq:strict-contact-pair-set} which is
$$
\mathfrak{Cont}_1^{\text{\rm st}}(M)  =  \bigcup_{\lambda \in \mathfrak{C}_1(M)} 
\{\lambda\} \times \Cont^{\text{\rm st}}(M,\lambda)
$$
as a subset of $\mathfrak{Diff}_1(M)$.
A priori, there is no apparent smooth structure with which we can equip the subset 
It is one of the fundamental steps to establish smoothness of this subset.

For this purpose,  we first introduce the following subset of 1-forms
$$
\Omega^1_\pi(M,\lambda) : = \{ \alpha^\pi \in \Omega^1(M) \mid \alpha^\pi 
= Z^\pi \intprod d\lambda, \, Z^\pi \in \Gamma(\ker\lambda)\}
$$
where $\pi = \pi_\lambda: TM \to \xi$ is the projection induced by $\lambda$
in terms of the splitting $TM = \xi \oplus \R \{R_\lambda\}$. Then we form the union
$$
\Omega_{\mathfrak C_1}^1(M) = \bigcup_{\lambda \in \mathfrak{C}_1(M)} \{\lambda \} \times
\Omega^1_\pi(M,\lambda).
$$
We then provide the set of strict contact pairs with a moduli representation 
which is a consequence of Proposition \ref{prop:C1-tangent} by considering the 
following map
$$
\Phi:  \mathfrak{Diff}_1(M,\mu) \to \Omega^1_{\mathfrak C_1}(M), \quad \Phi
=\{\Phi_\lambda\}_{\lambda \in \mathfrak{C}_1(M,\xi)}
$$
by 
\be\label{eq:Phi2lambda} 
\Phi(\lambda,\psi) = \Phi_\lambda(\psi) : = (\psi^*\lambda)^\pi, \quad \pi = \pi_\lambda.
\ee

We summarize the above preparation into the following lemma.
\begin{lem}\label{lem:st-contact-eq} We have 
$$
\mathfrak{Cont}_1^{\text{\rm st}}(M)  = \Phi^{-1}(0).
$$
\end{lem}
Therefore if we are given a contact manifold $M$, we can study the equation
\be\label{eq:reduced-equation}
(\psi^*\lambda)^\pi  = 0
\ee
of the  pair $(\psi,\lambda)$ on $\mathfrak{Diff}_1(M)$, i.e., with the constraint subset
of $\mathfrak{Diff}_1(M)$ consisting of $\psi$ satisfying
$$
\frac{d\psi^*\mu_\lambda}{d \mu_\lambda} = 0.
$$
This discussion motivates us to develop an off-shell framework of the 
equation starting from the set $\lambda$-incompressible diffeomorphisms.

 As the first step, we provide a precise representation of the tangent space of 
 $\Diff(M)$.
 
 \begin{lem}\label{lem:tangent-space} Let $\lambda$ be a contact form and $\psi$  a $\lambda$-incompressible
 diffeomorphism.  Suppose that $X$ be a vector field along $\psi$ tangent to $\Diff(M,\mu_\lambda)$.
 Then under the unique decomposition $X = X^\pi + k\, R_\lambda$, we have
 $$
 \nabla \cdot X = \nabla \cdot X^\pi + R_\lambda[k] = 0.
 $$
 \end{lem}
 \begin{proof} This follows from a straightforward  calculation
 $$
0 =  \CL_X(\mu_\lambda) = \nabla \cdot X\, \mu_\lambda = (\nabla \cdot X^\pi + R_\lambda[k])\,
\mu_\lambda.
$$
\end{proof}

At this point,  we would like to mention
 that any Reeb vector field  is divergence free. In fact the set of the vector fields tangent the Reeb 
flow provides a major obstacle to overcome in our proof of the main theorem.

\begin{lem} Let $X = k R_\lambda$ for a smooth function $k$. Then $\nabla\cdot X = 0$ if
and only if $R_\lambda[k] = 0$,
\end{lem} 

Among them the Reeb vector field corresponds to $k$ is a constant function $-1$, which
is a strict contactomorphism and cannot be eliminated for \emph{any} contact form.  
To remove this a priori presence of Reeb vector field, we will  restrict ourselves to
the vector fields of the type $X = X^\pi + k \, R_\lambda$ satisfying
\be\label{eq:mean0}
\int_M k\, d\mu_\lambda = 0
\ee
with respect to the contact form $\lambda$. We denote the set thereof by
\be\label{eq:set-mean0}
C^\infty_{(0;\lambda)}(M) : = \left\{f \in C^\infty(M,\R) \, \Big|\,  \int_M k\, d\mu_\lambda = \vol(M) \right\}.
\ee
This may be regarded as the tangent space of the quotient 
$$
\Diff(M,\mu_\lambda)/\Reeb(M,\lambda).
$$
This quotient space is not Hausdorff for general contact form $\lambda$. Because of this non-Hausdorffness 
we will not consider this quotient space. 

\section{The action of $\Diff(M,\mu_\lambda)$ on the strict contact pairs}

It is a standard fact that $\Diff(M,\mu_\lambda)$ is a (Frechet) smooth manifold 
modeled by the vector space of divergence-free vector fields.
The smooth structure $\Diff(M,\mu_\lambda)$
is described by Ebin-Marsden \cite{ebin-marsden}. 
Moreover the set of
contact forms is an open subset of $\Omega^1(M)$ and hence carries a smooth structure.

For the computational purpose, 
we equip $M$ with a Riemannian metric 
and the associated measure on $M$ for a given contact form $\lambda$. We will do this by taking
a contact triad $(M,\lambda,J)$ with a choice of CR-almost complex structure $J$ and
its associated triad metric considered in \cite{oh-wang:connection,oh-wang:CR-map1}.

\begin{rem}\label{rem:remark}
Throughout the rest of the paper, we will consider the natural pairing 
$$
\langle \cdot, \cdot \rangle: V \times V^* \to \R
$$
for a vector space (maybe infinite dimensional), and
the $L^2$-pairing between vector fields or one-forms on manifold $M$.
We mention that both pairings are (weakly) nondegenerate. In particular, in a suitable 
Banach completions of the latter vector spaces, we can apply  the Hahn-Banach theorem
in the argument of the Fredholm alternative in the study of cokernel via the consideration of
$L^2$-cokernel of the relevant operators appearing in our proofs later. 
(See \cite{oh:fredholm} for a similar practice applied to the study of 
Fredholm theory of Lagrangian boundary conditions of (pseudo-)holomorphic discs in
the Lagrangian Floer theory.)
\end{rem}

We recall that we have a natural pairing between one-forms $\alpha$ and vector fields $X$.
We say $X$ is an annihilator of a subspace $\CU \subset \Omega^1(M)$
if $\beta(X) = 0$ for all $\beta \in \CU$. Similar definition applies to one-form $\alpha$.
We apply this general practice to the action of $\Diff(M)$ on 
$$
\mathfrak{Cont}_1^{\text{\rm st}}(M) \subset \mathfrak{Diff}_1(M,\mu)
$$
which is the restriction of the obvious product action thereof on 
$\mathfrak{Diff}_1(M,\mu) = \mathfrak{C}_1(M) \times \Diff(M,\mu_\lambda)$.

When there is a symmetry of a Lie group $G$, the Fredholm study of 
the zero set of a $G$-invariant map to a vector space, we can decompose the tangent space
into the direct sum of the tangent space of the orbit and its orthogonal complement.
We apply this practice to the diagonal action of $\Diff(M,\mu_\lambda)$ on $\mathfrak{Diff}_1(M,\mu)$.
It will be also important to note that
 the equation of strict contact pairs carries an infinite dimensional symmetry
given by the diagonal action \eqref{eq:diff-action} of $\Diff(M,\mu_\lambda)$:
More specifically the map $\psi$ satisfies 
$$
\Phi ((\phi^{-1})^*\lambda,\phi \circ \psi) = 0 \Longleftrightarrow \Phi(\lambda,\psi) = 0
$$
for all diffeomorphism $\phi$, and in particular $\mathfrak{Cont}_1^{\text{\rm st}}(M)$ is invariant 
under this diagonal action of $\Diff(M,\mu_\lambda)$. 
Utilizing this symmetry plays a crucial role in our Fredholm analysis of the map $\Pi$.

As the first step,  we provide the description \emph{in the formal level} in 
the $C^\infty$ spaces,
and justify the calculations in a suitable Banach completions in the next subsection.

Similarly as in \cite{freed-uhlen}, which concerns the slice theorem for the action of gauge group
on the space of anti-self dual connections, we will take the 
(infinitesimal) slice of the action \eqref{eq:diff-action} obtained by taking the
$L^2$-orthogonal in $\mathfrak{Cont}_1^{\text{\rm st}}(M)$ to the $\Diff(M,\mu_\lambda)$-orbit passing through a strict contact pair  $(\lambda,\psi) \in \mathfrak{Cont}_1^{\text{\rm st}}(M)$. For this purpose
we first compute the tangent space of the orbit at $(\lambda,\psi)$ by recalling the derivatives
$$
\frac{d}{dt}\Big|_{t=0} ((\exp (-tX))^*\lambda -\lambda, \exp (tX)\psi) = \big( -\CL_X \lambda, (X \circ \psi)\big)
$$
remembering the equation $\psi^*\lambda = \lambda$. 

We then equip $\Diff(M,\mu_\lambda)$ the
$L^2$-metric on the tangent space of $\Diff(M,\mu_\lambda)$ at the identity and then at arbitrary $\psi$ by 
right translations. The $L^2$-orthogonal space is the set 
of $(Y \circ \psi, (\psi^{-1})^*\alpha)$ satisfying 
$$
\int_M \langle X,Y \rangle \, d\vol - \int_M \langle \CL_X\lambda, \alpha \rangle \, d\vol = 0
$$
for all choices of $(\alpha,X)$. By considering 
 the isomorphism $Z \mapsto \langle Z, \cdot \rangle_2$,
we instead consider the equation of the pair 
$$
(\alpha, Y) \in (\Diff(M,\mu_\lambda) \cdot (\lambda,\psi))^\perp \subset  \Omega^1(M) \times \Gamma(TM).
$$ 
 (See \cite{oh:fredholm} for a similar practice in the parametric 
transversality study of  pseudoholomorphic curves under the perturbation of 
Lagrangian boundary condition.)

\begin{prop}\label{prop:orthogonal} Let $(\alpha,Y) \in 
 (\Diff(M,\mu_\lambda) \cdot (\lambda,\psi))^\perp \subset  \Omega^1(M) \times \Gamma(TM)$. Express $Y$ as
 \be\label{eq:Y}
 Y = Y^\pi + g R_\lambda, \quad \nabla\cdot Y = 0.
 \ee
Then we have
\be\label{eq:dalpha}
d\alpha = d(Y^\pi \rfloor d\lambda).
\ee
\end{prop}
\begin{proof} By the definition of annihilator, we have
\be\label{eq:dual}
\int_M \alpha(X)\, d\vol - \int_M \CL_X\lambda(Y)\, d\vol = 0
\ee
for all $X$ satisfying $\nabla \cdot X = 0$. The second integral can be rewritten as
\bea\label{eq:intLXlambdaY}
\int_M \CL_X\lambda(Y)\, d\vol & = & \int_M (d(X \rfloor \lambda) + X \rfloor d\lambda)(Y)\, d\vol 
\nonumber\\
& = & \int_M (dg -  (Y \rfloor d\lambda))(X)\,d\vol.
\eea
Recall $\lambda(Y) = g$.
Therefore by substituting \eqref{eq:intLXlambdaY} into \eqref{eq:dual}, 
\beastar
0 & = & \int_M \alpha(X)\, d\vol - \int_M (dg -  (Y \rfloor d\lambda))(X)\,d\vol \\
& =&  \int_M (\alpha- (dg -  (Y \rfloor d\lambda))(X)\,d\vol 
\eeastar
for all $X$  satisfying $\nabla \cdot X = 0$. Therefore we have derived
that 
$$
d(\alpha- (dg -  (Y \rfloor d\lambda)) = 0,
$$
which is equivalent to
$$
d\alpha = d(Y^\pi \rfloor d\lambda).
$$
 This finishes the proof.
\end{proof}

An immediate corollary of this proposition with $\alpha = 0$ is the following.

\begin{cor}\label{cor:alpha=0}
Let $Y \in (\Diff(M,\mu_\lambda) \cdot (\psi))^\perp \subset\Gamma(TM)$ with
$$
 Y = Y^\pi + g R_\lambda, \quad \nabla\cdot Y = 0.
$$
Then we have
\be\label{eq:Deltag}
d(Y^\pi \rfloor d\lambda) = 0.
\ee
Equivalently we have $\nabla \cdot Y^\pi = 0$.
\end{cor}

We summarize the discussion of the present section to the following morphism of 
 exact sequences:  
\be\label{eq:morphism-sequences}
\xymatrix{ 0 \ar[r] & T_{(\lambda,\psi))} \CO_{\Diff(M,\mu_\lambda)}(\lambda,\psi) \ar[r] \ar[d] 
& T_{(\lambda,\psi)}\mathfrak{Diff}_1(M,\mu))
\ar[r] \ar[d] &  (T_{(\lambda,\psi)} \CO_{\Diff(M,\mu_\lambda)})^\perp \ar[r] \ar[d]& 0 \\
0 \ar[r] &  T_\lambda \CO_{\Diff(M,\mu_\lambda)}  \ar[r] 
&  T_\lambda \mathfrak{C}_1(M) 
\ar[r] & (T_\lambda \CO_{\Diff(M,\mu_\lambda)} )^\perp \ar[r] & 0.
}
\ee
Here we simplify the notations for the $\Diff(M,\mu_\lambda)$-orbits of $(\lambda,\psi)$ or of $\lambda$
to $\CO_{\Diff(M,\mu_\lambda)}$ which should not confuse readers since they appear together with
the points associated to the orbits at which we take the tangent spaces thereof.

Since the first two vertical arrows are surjective, the third is also surjective by 
the Five Lemma.

\section{Smoothness of {$ \mathfrak{Cont}_1^{\text{\rm st}}(M)$}}
\label{sec:smoothness}

By the definition of the map
\be\label{eq:Phi}
\Phi:  \mathfrak{Diff}_1(M,\mu)  \to  \Omega^1_{\mathfrak C_1}(M),
\ee
which is given by  \eqref{eq:Phi2lambda}, we have
$$
\Phi^{-1}(0) = \mathfrak{Cont}_1^{\text{\rm st}}(M) = \bigcup_{\lambda \in \mathfrak{C}_1(M)}
\{\lambda \} \times \Cont^{\text{\rm st}}(M,\lambda).
$$
 Let $(\lambda,\psi) \in \Phi^{-1}(0)$, i.e., it satisfy $\psi^*\lambda = \lambda$. 
Then we have 
\be\label{eq:DPhi2}
D\Phi(\lambda,\psi)(\alpha, X) = \delta_{(\alpha,X)}\left(\psi^*\lambda)^\pi \right)
\ee
the first variation of $\Phi$ along the tangent vector $(\alpha,X)$ with
\be\label{eq:alpha-variation}
\alpha = Z^\pi_\alpha \intprod d\lambda + h_\alpha \, \lambda.  
\ee
We recall readers that $\int_M h_\alpha \, d\mu_\lambda = 0$.

We denote by $\Pi = \Pi_\lambda: TM \to TM$ be the idempotent associated to the 
projection $\pi_\lambda$. By an abuse of notation, we also denote by $\Pi$ the induced 
map on $T^*M$. 

\begin{prop} \label{prop:variation-pi}
At any contact pair $(\lambda,\psi)$, Let $\{(\lambda_t,\psi_t)\}_{-\epsilon < t < \epsilon}$ be the germ of 
a smooth curve at $(\lambda,\psi)$ representing the variation $(\alpha,X)$. Then we have
$$
 \delta_{(\alpha,X)} \left((\psi^* \lambda)^\pi\right) 
 =  B_\alpha (\psi^*\lambda) + \Pi (\psi^*\alpha)
 +  \Pi (\psi^*\CL_X \lambda)
 $$
 where $B_\alpha: = \delta_\alpha(\Pi_\lambda)$.
\end{prop}
\begin{proof}
We compute
$$
\delta_{\alpha}\left( (\psi^*\lambda)^\pi\right) = \delta_\alpha(\Pi_\lambda)(\psi^*\lambda)
+ \Pi_\lambda(\psi^*\alpha).
$$
By definition, we have
$$
\delta_\alpha(\Pi_\lambda)= \frac{d}{dt}\Big|_{t=0}\Pi_{\lambda_t}
$$
for the germ of a curve $\lambda_t$ with $-\epsilon < t < \epsilon$ satisfying $\lambda_0 = \lambda$ and $\dot \lambda_t|_{t=0} = \alpha$.

For the clarity of the exposition, we set the endomorphism of $TM$
$$
B_\alpha: = \delta_\alpha \Pi_\lambda.
$$
This in turn gives rise to
\be\label{eq:delta-alpha}\delta_{\alpha} (\psi^*\lambda)^\pi
= B_\alpha (\psi^*\lambda) + \Pi_\lambda (\psi^*\alpha).
\ee

On the other hand,  we compute
$$
\delta_X(\psi^*\lambda)^\pi =  \Pi_\lambda(\psi^*\CL_X\lambda).
$$
By combining the two, we have derived
\be\label{eq:delta-X}
\delta_{(\alpha,X)} \left((\psi^* \lambda)^\pi\right) =
 B_\alpha(\psi^*\lambda) + \Pi (\psi^*\alpha) +
\Pi(\psi^*\CL_X\lambda).
\ee
In particular, if $\psi^*\lambda = \lambda$, the formula is reduced to
$$
 B_\alpha(\lambda) + \Pi (\psi^*\alpha) + 
\Pi (\psi^*\CL_X\lambda) .
$$
  Combining the above calculations, we have proved
$$
 \delta_{(\alpha,X)} \left((\psi^* \lambda)^\pi\right) = B_\alpha(\lambda)
+ \Pi (\psi^*\alpha) +  \Pi (\psi^*\CL_X \lambda).
$$
 Here the last equality arises since $\Pi_\lambda(\lambda) = 0$.
 \end{proof}

Therefore at a strict contact pair $(\lambda,\psi)$, we have
$$
B_\alpha (\psi^*\lambda) = B_\alpha (\lambda).
$$
In fact, we have the following explicit formula whose
derivation will be postponed until Appendix.

\begin{lem}\label{lem:Balphalambda} Let $\alpha = Z^\pi \intprod d\lambda + h\, \lambda$. Then
$$
B_\alpha(\lambda) = - Z^\pi \intprod d\lambda - (dh)^\pi.
$$
In particular if $(\lambda,\psi)$ is a strict contact pair, then
\be\label{eq:deltaalphaX}
\delta_{(\alpha,X)} \left((\psi^* \lambda)^\pi\right) 
 = (- Z^\pi  + \psi_*Z^\pi) \intprod d\lambda +  \Pi (\psi^*\CL_X \lambda) - (dh)^\pi.
 \ee
 \end{lem}

Combining the above discussion, we are now ready to give the proof of Lemma 
\ref{lem:bounded-DPhi} which was postponed until  the preset section.

\begin{proof}[Proof of Lemma \ref{lem:bounded-DPhi}] By staring at the formulae
\eqref{eq:DPhi2} and \eqref{eq:deltaalphaX}, we find that the image of the operator
 $D\Phi(\lambda,\psi)$
will be contained in $C^\varepsilon$, provided $(\alpha,X)$ is in $C^{(1,\varepsilon)}$,
and that there exists a uniform constant $C = C(\psi,\lambda) > 0$ depending only on
$C^1$-norm of $(\psi,\lambda)$ such that
$$
\|D\Phi(\lambda,\psi)(\alpha,X)\|_\varepsilon \leq C \|(\alpha,X)\|_{1,\varepsilon}.
$$
This proves the boundedness of the linear operator $D\Phi(\lambda,\psi)$.
\end{proof}

With these preparations, we are now ready to give the proof of the following
submersion property. 
The proof of this theorem is an interesting amalgamation of the Riemannian geometry, 
more specifically the Hodge theory and of the contact geometry, of the contact triad $(M,\lambda,J)$.
Actually such an interplay should be anticipated because consideration of the divergence of a vector field
belongs to the realm of the Riemannian geometry and consideration of distribution $\xi$ or 
the nondegenerate one-form $\lambda$ belongs to that of the contact geometry.

\begin{thm}\label{thm:Phi-submersion} The map $\Phi$ is a submersion on 
$\mathfrak{Cont}_1^{\text{\rm st}}(M,\Omega)$. In particular, 
$$
\mathfrak{Cont}_1^{\text{\rm st}}(M) \subset  \mathfrak{Diff}_1(M,\Omega)
$$
is an (infinite dimensional) smooth Frechet submanifold.
\end{thm}
\begin{proof} Let $(\lambda,\psi)$ be a strict contact pair.
Then the formula  \eqref{eq:deltaalphaX} for $\alpha$ with the decomposition \eqref{eq:alpha-variation}
  is  reduced to
\be\label{eq:DPhi-strict}
D\Phi(\lambda,\psi)(\alpha, X) = (-Z^\pi  + \psi_*Z^\pi) \intprod d\lambda +  \psi_*(X^\pi) \intprod d\lambda - (dh)^\pi
\ee
Suppose  $\eta = W^\pi \intprod d\lambda \in \Omega^1_\pi(M,\lambda)$
satisfies
\be\label{eq:coker-dPhi}
0 =  \int_M\left \langle \eta,  \left( (-Z^\pi  + \psi_*Z^\pi) \intprod d\lambda + \psi_*(X^\pi) \intprod d\lambda  
- (dh)^\pi\right) \right \rangle \,  d\mu_\lambda
\ee
for all $(\alpha, X)$ satisfying $\nabla\cdot X = 0$.

Recalling Lemma \ref{lem:tangent-space}, we know that
 $$
 \nabla\cdot X = 0 \Longleftrightarrow \nabla \cdot X^\pi + R_\lambda[\lambda(X)] = 0.
 $$
 In particular, we can set $X = 0$. Then   \eqref{eq:coker-dPhi} is reduced to
$$
0 =  \int_M \left \langle \eta, 
 \left(\psi^*(- \psi_*Z^\pi + Z^\pi) \intprod d\lambda \right) - (dh)^\pi \right\rangle \, 
 d\mu_\lambda
$$
for all $\alpha = Z^\pi \intprod d\lambda + h\, \lambda$. By the integration by parts 
and using the equation $\psi^*\lambda = \lambda$,
we rewrite this equation into
\beastar
0 & = & \int_M (- \psi^*\eta +\eta)( Z^\pi) - \langle \eta, (dh)^\pi \rangle \,  d\mu_\lambda \\
& = & \int_M (- \psi^*\eta +\eta)( Z^\pi) - \langle \eta, dh \rangle \,  d\mu_\lambda
\eeastar
for all $\alpha$. Since $\eta = Z^\pi \intprod d\lambda = \eta^\pi$, this gives rise to
\be\label{eq:etapi-eta}
\psi^*\eta = \eta, \quad \& \quad \delta \eta = 0
\ee 
where $\delta$ is the Hodge codifferential.

On the other hand, setting $\alpha = 0$,  \eqref{eq:coker-dPhi} becomes
$$
0 =  \int_M \langle \eta, \psi_*(X^\pi)
\intprod d\lambda \rangle \, d\mu_\lambda
=   \int_M \langle \psi_*\eta, X^\pi\intprod d\lambda \rangle \, d\mu_\lambda
$$
for all divergence-free vector field $X$. Using $\psi^*\eta = \eta$,  $\eta^\pi = W^\pi \intprod d\lambda$
for some vector field $W$ and $d\psi (\xi) = \xi$, we have 
$$
(\psi_*\eta)^\pi = \eta^\pi = W^\pi  \intprod d\lambda.
$$
This enables us to rewrite
$$
 \int_M \langle \psi_*\eta, X^\pi\intprod d\lambda \rangle \, d\mu_\lambda
 = \int_M d \lambda(JW^\pi, X^\pi)\, d\lambda
 $$
 and hence we obtain
 \be\label{eq:L2-innerproduct=0}
 \int_M - d \lambda(JW^\pi, X^\pi)\, d\lambda = 0
 \ee
 for all vector fields $X$ satisfying $\nabla \cdot X=0$.  We now prove the following.
 
\begin{lem}\label{lem:ampleness}  The one-form $- J W^\pi \intprod d\lambda$ is closed.
\end{lem}
\begin{proof}
 We convert the vector fields to
one-forms by the index lowering operation
$$
\alpha_X: = X^\flat  = \langle X, \cdot \rangle = - d\lambda(JX,\cdot) \oplus \lambda(X) \, \lambda
$$
using the triad metric associated to the contact triad $(M,\lambda, J)$.
Then it is a standard fact in the Hodge theory on Riemannian manifold that `taking the divergence
of $X$' corresponds to `taking the Hodge codifferential $\delta$ of the one-form $\alpha_X$'
under the  index lowering operation
Then \eqref{eq:L2-innerproduct=0} can be converted to the integral
$$
\int \langle (-JW^\pi \intprod d\lambda), \beta \rangle \, d\vol = 0
$$
for all one-form $\beta$ satisfying $\delta \beta = 0$. 

For the simplicty of exposition, we set $\overline W = -JW$. Then we have $\overline W^\pi = -J W^\pi$.
Therefore $\overline W^\pi \intprod d\lambda$
is $L^2$-orthogonal to $\delta \Omega^2(M)$. By the Hodge decomposition
$$
\Omega^1(M) = H^1_g(M) \oplus B_1(M) \oplus \delta (\Omega^2(M))
$$
it implies that $\overline W^\pi \intprod d\lambda \in H^1_g(M) \oplus B_1(M)$ and hence
$$
\overline W \intprod d\lambda = \overline W^\pi \intprod d\lambda =  \zeta + dk
$$
for a harmonic one-form $\zeta$ and a smooth function $k \in C^\infty(M,\R)$.
In particular, 
$$
d(\overline W \intprod d\lambda) = d(\overline W^\pi \intprod d\lambda) = 0.
$$
This finishes the proof.
\end{proof}
 Together with the second equation of \eqref{eq:etapi-eta}, this implies
that \emph{$\eta^\pi$ is a harmonic one-form}.

Since $\overline W^\pi \in \xi$, this closedness is equivalent to
 $$
 0 = d(\CL_{\overline W^\pi}\lambda) = \CL_{\overline W^\pi} d\lambda.
 $$
Denote  the flow of $\overline W^\pi$ by $\phi_t$ with $\phi_0 = \id$. Then we obtain
$$
0 = d\left(\frac{d}{dt} \phi_t^*\lambda\right).
$$
This implies $d(\phi_t^*\lambda -\lambda) = 0$ for all $t$. Write $\gamma_t: = \phi_t^*\lambda - \lambda$.
Then $\gamma_0 = 0$ and hence $\gamma_0$ is exact. This implies
$\phi_t^*\lambda - \lambda$ is exact for all $t$. Therefore 
$$
\overline W^\pi \intprod d\lambda = \CL_{\overline W^\pi}\lambda 
= \frac{d}{dt}\Big|_{t=0}(\phi_t^*\lambda - \lambda)
$$
is exact. This then implies the harmonic one-form $\eta^\pi = \overline W^\pi \intprod d\lambda$
must be zero. This proves 
$$
\coker D\Phi(\lambda,\psi) = \{0\}
$$
and hence $D\Phi(\lambda,\psi)$ is surjective at all $(\lambda,\psi) \in \mathfrak{Cont}_1^{\text{\rm st}}(M)$.
This finishes the proof.
\end{proof}

\section{Fredholm analysis of the projection $\Pi$}
\label{sec:ker-cokerPi1}

Now we consider the projection map
$$
\Pi: \Phi^{-1}(0) \to \mathfrak{C}_1(M)
$$
where we have $\Phi^{-1}(0) = \mathfrak{Cont}^{\text{\rm st}}_1(M,\lambda)$ by definition.
To put the study of this map $\Pi$ into a Banach manifold framework, we proceed as Floer did \cite{floer:unregularized}.
(See \cite{wendl:super-rigidity} too.)
Let $\lambda$ be any given contact form.
We introduce Floer's $C^\varepsilon$ perturbation Banach bundles
$$
\bigcup_{\lambda \in \mathfrak{C}_1(M)} \{\lambda\} \times \mathfrak{X}^{(1,\varepsilon)}(M,\mu_\lambda)
\subset T \mathfrak{Diff}_1(M,\mu_\lambda)
$$
and consider the perturbations of $(\lambda,\psi)$ of the form
$$
(\exp_\lambda(\alpha), \exp_\psi(X)).
$$
Then we consider  a $C^{(1,\varepsilon)}$-neighborhood of $\lambda$ in
$\mathfrak{C}_1(M)$ and and work with the restriction
$$
\Phi|_{ \CU_\varepsilon(\mathfrak{Diff}_1(M,\mu))}: \CU_\varepsilon(\mathfrak{Diff}_1(M,\mu)) \to \Omega^1_{\mathfrak C_1}(M).
$$
We then consider the restriction of $\Pi$ to 
$$
\Pi: \Phi^{-1}(0) \cap \CU_\varepsilon(\mathfrak{Diff}_1(M,\mu)) \to \mathfrak C_1(M).
$$
By definition, we have
$$
\Phi^{-1}(0) \cap \CU_\varepsilon(\mathfrak{Diff}_1(M,\mu)) = 
\bigcup_{\lambda' \in \CU_\varepsilon(\mathfrak{C}_1(M);\lambda)}  \{\lambda'\} \times \Diff^{(1,\varepsilon),\text
{\rm st}}(M,\lambda).
$$
We put
$$
\CU_\varepsilon(\mathfrak{Cont}_1^{\text{\rm st}}(M); (\lambda,\psi)) : = \Phi^{-1}(0) \cap \CU_\varepsilon(\mathfrak{Diff}_1(M,\mu)).
$$
\begin{notation}\label{nota:simplification}
\emph{For the notational simplicity, we will just write
$$
\Pi:  \mathfrak{Cont}_1^{(1,\varepsilon);\text{\rm st }}(M) \to 
 \mathfrak{C}^{(1,\varepsilon)}(M)
 $$
 for 
 $$
\CU_\varepsilon(\mathfrak{Cont}_1^{\text{\rm st}}(M); (\lambda,\psi)) 
\to \CU_\varepsilon(\mathfrak{C}_1(M);\lambda)
$$
in the rest of the present section.}
\end{notation}

\begin{thm}\label{thm:quotient-manifold} Assume $\lambda$ is non-projectible. Consider
the projection 
$$
\Pi:  \mathfrak{Cont}_1^{(1,\varepsilon);\text{\rm st }}(M) \to 
 \mathfrak{C}_1^{(1,\varepsilon)}(M).
 $$
 Then at each $\lambda \in \mathfrak{C}_1(M)$,  the derivative
 $$
 d\Pi: T_{[(\lambda,\psi)]} \mathfrak{Cont}_1^{(1,\varepsilon);\text{\rm st }}(M) \to T_\lambda \mathfrak{C}_1^{(1,\varepsilon)}(M) 
$$
is a Fredholm map of index 1 at all $\lambda \in \mathfrak{C}_1^{\text{\rm np}}(M)$.
\end{thm}
The proof of this theorem will occupy the rest of the present section. 

\subsection{Analysis of the kernel of the linearization}

We consider the map
$$
\Pi: \mathfrak{Cont}_1^{(1,\varepsilon);\text{\rm st }}(M) \to  \mathfrak{C}_1^{(1,\varepsilon)}(M).
$$
Its tangent map has the expression
$$
d\Pi: T_{[(\lambda,\psi)]} \mathfrak{Cont}_1^{(1,\varepsilon);\text{\rm st }}(M) \to T_\lambda \mathfrak{C}_1^{(1,\varepsilon)}(M) 
\cong \Omega^{1,(1,\varepsilon)}_{\mathfrak C}(M)
$$
which is  the projection $(\alpha,X) \mapsto \alpha$ restricted to
the pairs 
$$
(\alpha,X) \in T_{(\lambda,\psi)} \mathfrak{Cont}_1^{(1,\varepsilon);\text{\rm st }}(M).
$$
 
\begin{prop}\label{prop:dPi-kernel} Assume $\lambda$ is non-projectible. 
The map $d\Pi(\lambda,\psi)$ has 
$$
\ker d\Pi(\lambda,\psi) = \{(0, c R_\lambda) \mid c \in \R\}
$$
for all $(\lambda,\psi) \in \mathfrak{Cont}^{1,\epsilon;\text{\rm st }}(M)$.
\end{prop}
\begin{proof} Since 
$$
(\alpha,X) \in T_{(\lambda,\psi)} \mathfrak{Cont}^{1,\epsilon;\text{\rm st }}(M),
$$
the regularity of the value $0$ of $\Phi$ and the preimage value theorem implies that $(\alpha,X)$ 
satisfies 
$$
(\psi_*(X^\pi + Z^\pi) - Z^\pi)\intprod d\lambda -(dh)^\pi = 0
$$
which is the linearization of $\Phi$ written in terms of the decomposition
\be\label{eq:alpha}
\alpha = Z^\pi \intprod d\lambda + h\, \lambda.
\ee
The equation is equivalent to
\be\label{eq:linearization-general}
(X^\pi + Z^\pi -\psi^*Z^\pi) \intprod d\lambda - (dh)^\pi = 0.
\ee
We decompose 
$$
X = X^\pi + k \, R_\lambda.
$$
Then by the nondegeneracy of $d\lambda$ on $\xi$, we derive
$$
X^\pi =   - Z^\pi + \psi^*Z^\pi - X_h^\pi
$$
recalling the decomposition $dh = X_h^\pi \intprod d\lambda + R_\lambda[h]\, \lambda$.
Therefore if $\alpha = 0$ and hence $Z^\pi = 0$ and $h=0$, we have obtained
$$
X^\pi = 0, \quad R_\lambda[k] = 0
$$
where the latter equation follows from 
$$
0 =\nabla\cdot X =  \nabla \cdot (X^\pi + k \, R_\lambda) = \nabla\cdot X^\pi + R_\lambda[k]
$$
combined with $X^\pi = 0$. Therefore we have derived 
\be\label{eq:X=kR}
X = k \, R_\lambda, \quad R_\lambda[k] = 0.
\ee
By the nonprojectibility of $\lambda$ given in Definition \ref{defn:nonprojectible}, $dk = 0$.
 This finishes the proof.
\end{proof}

\subsection{Cokernel vanishing of the linearization}

Next we study the cokernel of the map $\Pi$. 

\begin{prop}\label{prop:dPi-cokernel}  Assume $\lambda$ is non-projectible.  Then we have 
$$
\text{\rm Coker } d\Pi(\lambda,\psi) = \{0\}.
$$
\end{prop}
\begin{proof}
Consider the operator
$$
d\Pi_{(\lambda,\psi)}: \ker D\Phi_{(\lambda,\psi)} \to T_\lambda \mathfrak{C}_1^{(1,\varepsilon)}(M) 
$$
the codomain of which is a completion of $\Omega^1_\lambda(M)$.
We will examine the kernel of the $L^2$-dual  $(d\Pi_{(\lambda,\psi)})^\dagger$ of 
the operator
\be\label{eq:dPi}
d\Pi_{(\lambda,\psi)}: \ker D\Phi_{(\lambda,\psi)} \to T_\lambda \mathfrak{C}_1^{(1,\varepsilon)}(M).
\ee
We first note that by the $\Diff(M,\mu_\lambda)$-symmetry, $\beta$ must be orthogonal to the orbit
$\CO_{\Diff(M,\mu_\lambda)}(\lambda)$, i.e., 
$$
\beta \in T_\lambda(\CO_{\Diff(M,\mu_\lambda)}(\lambda))^{\perp_2}.
$$
In particular we have 
\be\label{eq:nablaY=0}
\nabla \cdot Y = 0
\ee
by Corollary \ref{cor:alpha=0}.

Suppose that $\beta \in T_\lambda \mathfrak{C}_1(M)$ satisfies
\be\label{eq:L2kernel-eq}
 \int_M \langle (d\Pi_{1,(\lambda,\psi)})(\alpha,X), \beta \rangle  \, d\mu_\lambda = 0
\ee
for all $(\alpha,X)$.  
We again consider the natural pairing $(\beta,Y) \mapsto \beta(Y)$
instead of the $L^2$-paring $(\alpha,\beta) \mapsto \langle \alpha, \beta \rangle$ on the
codomain 
$$
T_\lambda \mathfrak{C}_1^{(1,\varepsilon)}(M) \cong \Omega_\pi^{(1,\varepsilon)}(M)
$$
in the discussion below: The precise relationship between $\beta$ and
its flat $\beta^\sharp = Y$ is given by
\be\label{eq:beta}
\beta = - JY^\pi \intprod d\lambda + \beta(Y)\, \lambda.
\ee
Then \eqref{eq:L2kernel-eq} is equivalent to
$$
\int_M  d\Pi_{1,(\lambda,\psi)}(\alpha,X)(-JY) + d\Pi_{(1,(\lambda,psi)}(\alpha,X)(\beta(Y)R_\lambda)
\, d\mu_\lambda = 0
$$
for all $(\alpha,X) \in \ker D\Phi_{1,(\lambda,\psi)}$. 
Then since $(\alpha,X) \in \ker D\Phi(\lambda,\psi)$, it satisfies  
\be\label{eq:alphaX-kernel}
0 = \left((-Z^\pi + \psi_*Z^\pi)  + \psi_*(X^\pi)\right) \intprod d\lambda - X_h^\pi \intprod d\lambda
\ee
by the formula \eqref{eq:deltaalphaX} of $D\Phi_{(\lambda,\psi)}$
under the decomposition
$\alpha = Z^\pi \intprod d\lambda + h\, \lambda$. 

Clearly the pair $(\alpha, X)$ given by
\be\label{eq:choice1}
X = -Z^\pi + \psi_*Z^\pi, \quad \alpha = Z^\pi  \intprod d\lambda
\ee
satisfies the last equation, \emph{provided $\nabla\cdot X = 0$}.

\begin{lem} We have $\nabla\cdot X = 0$.
\end{lem}
\begin{proof} We have only to prove $\CL_X\mu_\lambda = 0$. 
\emph{We here emphasize the standing hypothesis $\psi^*\lambda = \lambda$ and
hence $\psi_*Z^\pi \in \xi$ because we have
$$
\lambda(\psi_*Z^\pi) = \psi^*\lambda(Z^\pi) = \lambda(Z^\pi) = 0.
$$
In particular $X \in \xi$.}
Using this, we compute
$$
\CL_X \mu_\lambda = d(X \intprod \lambda \wedge (d\lambda)^n)
= d(\lambda(X)) \wedge (d\lambda)^n.
$$
Then we compute $\lambda(X) = \lambda (-Z^\pi + \psi_*Z^\pi) = 0$.
This proves $\CL_X\mu_\lambda = 0$.
\end{proof}

Therefore by considering such a pair $(X,\alpha)$,  the equation \eqref{eq:L2kernel-eq}
becomes
\beastar 
0 &= &  \int_M  d\Pi_{(\lambda,\psi)}(\alpha,X)(-JY + \beta(Y) R_\lambda)\, d\mu_\lambda \\
& = & \int_M\alpha(-J Y^\pi)  \, d\mu_\lambda = \int_Md\lambda(Z^\pi, -JY^\pi)  \, d\mu_\lambda 
\eeastar
for all $Z^\pi$. This implies $JY^\pi = 0$ and hence 
\be\label{eq:Ypi=0}
Y^\pi = 0.
\ee
Therefore we have $Y = h \, R_\lambda$. Then since $\nabla\cdot Y = 0$ from
\eqref{eq:nablaY=0}, we obtain
$$
0 = \nabla \cdot Y = \nabla \cdot(h \, R_\lambda) = R_\lambda[h].
$$
Since $\lambda$ is non-projectible, we have $dh = 0$ again by definition in Definition  \ref{defn:nonprojectible}.
This implies $\beta = c \lambda$ for some constant $c$. Since any $\lambda \in \mathfrak{C}_1(M)$
has fixed volume $\vol(\lambda) = 1$, we must have $c = 0$. This finishes the proof.
\end{proof}

An immediate corollary of the above calculations, combined with the obvious $\R$-action invariance
of $\Cont^{\text{\rm st}}(\lambda)$ induced by the Reeb flows, is the following.

\begin{cor}\label{cor:fibration} The set of strict contactomorphisms 
$$
\Cont^{\text{\rm st}}(\lambda)
$$
of any non-projectible contact form $\lambda$ is a smooth one-dimensional manifold 
consisting of a disjoint union of real lines $\R$, one
for each connected component generated by the Reeb flow.
\end{cor}
This finishes the proof of  Theorem \ref{thm:nostrict-intro}.

\appendix

 \section{Proof of Lemma \ref{lem:Balphalambda}}

In this section, we will use Dirac's bracket notations \cite[Chapter II]{dirac}: $|\cdot \rangle$ denotes
a vector and $\langle \cdot|$ denotes a covector, and $\langle \cdot|\cdot \rangle$
denotes an inner product and $|\cdot \rangle\langle \cdot|$ denotes the 
tensor product 
$$
|\cdot \rangle \otimes \langle \cdot| \in V^*\otimes V \cong \Hom(V,V).
$$
We fix a Darboux frame
$$
\{E_1, \cdots, E_n, JE_1, \cdots JE_n, R_\lambda\}
$$
in a neighborhood of any given point $x$ for the given contact form $\lambda$ of $(M,\xi)$.

Let $\alpha = Z^\pi \intprod d\lambda + h\, \lambda$ be given. We then have
$$
\delta_\alpha \Pi = \delta_{\alpha^\pi} \Pi + \delta_{h \lambda} \Pi.
$$
We will compute the two summands separately.

Then for a one-form  of the type
$$
\alpha = Z^\pi \intprod d\lambda = \CL_{Z^\pi} \lambda, 
$$
we derive
$$
B_\alpha =  \delta_\alpha \Pi_\lambda = \CL_{Z^\pi}(\Pi_\lambda) 
= \sum_{i=1}^n \CL_{Z^\pi} (\left|E_i \rangle \langle E_i\right|) 
+  \sum_{i=1}^n \CL_{Z^\pi} (\left| JE_i \rangle \langle J E_i\right|).
$$
We compute
\beastar
\sum_{i=1}^n \CL_{Z^\pi} (\left|E_i \rangle \langle E_i\right|) 
& = & \sum_{i=1}^n  \left|\CL_{Z^\pi} E_i \rangle \langle E_i\right| 
 +\left|E_i \rangle \langle \CL_{Z^\pi} E_i\right|\\
& = & - \sum_{i=1}^n  \left(\left|\CL_{E_i} Z^\pi \rangle \langle E_i\right|  +
\left|E_i \rangle \langle \CL_{E_i} Z^\pi\right|\right)
\eeastar
and then
\beastar
\sum_{i=1}^n \CL_{Z^\pi} (\left|E_i \rangle \langle E_i\right|) (\lambda)
& = & 
= - \sum_{i=1}^n  \langle R_\lambda \left|\CL_{E_i} Z^\pi \rangle \langle E_i\right|  +
\langle R_\lambda \left |E_i \rangle \langle \CL_{E_i} Z^\pi\right|\\
& = & - \sum_{i=1}^n  \langle R_\lambda \left|\CL_{E_i} Z^\pi \rangle \langle E_i\right| 
= - \sum_{i=1}^n \lambda(\CL_{E_i} Z^\pi) \langle E_i| \\
& = & \sum_{i=1}^n (\CL_{E_i} \lambda)(Z^\pi) \langle E_i| = \sum_{i=1}^n (E_i \intprod d\lambda)(Z^\pi)
\langle E_i| \\
& = &  \sum_{i=1}^n (d\lambda(E_i, Z^\pi) \langle E_i|  = - \sum_{i=1}^n (d\lambda(Z^\pi, E_i) \langle E_i|.
\eeastar
Similar computation leads to
$$
\sum_{i=1}^n \CL_{Z^\pi} (\left|JE_i \rangle \langle JE_i\right|) (\lambda)
= -\sum_{i=1}^n (d\lambda(Z^\pi, JE_i) \langle JE_i|.
$$
By adding up the two, we have derived
$$
B_\alpha(\lambda) = - \sum_{i=1}^n (d\lambda(Z^\pi, E_i) \langle E_i| -
\sum_{i=1}^n (d\lambda(Z^\pi, JE_i) \langle JE_i|  = - Z^\pi \intprod d\lambda = -\alpha.
$$

Next consider the one-form $\alpha = h \lambda$. Similarly we compute
$$
\CL_{hR_\lambda} \Pi = \sum_{i=1} \CL_{hR_\lambda}(|E_i\rangle \langle E_i|) + \CL_{hR_\lambda}(|JE_i \rangle \langle JE_i|)
$$
and
\beastar
\sum_{i=1} \CL_{hR_\lambda}(|E_i\rangle \langle E_i|) = 
& = & \sum_{i=1}^n  \left|\CL_{hR_\lambda} E_i \rangle \langle E_i\right| 
 +\left|E_i \rangle \langle \CL_{hR_\lambda} E_i\right|\\
& = & - \sum_{i=1}^n  \left(\left|\CL_{E_i} (hR_\lambda) \rangle \langle E_i\right|  +
\left|E_i \rangle \langle \CL_{E_i}(hR_\lambda)\right|\right).
\eeastar
We compute
$$
\CL_{E_i}(hR_\lambda) = dh(E_i) R_\lambda + h \CL_{E_i} R_\lambda
$$
and hence
\beastar
\sum_{i=1} \CL_{hR_\lambda}(|E_i\rangle \langle E_i|)(\lambda) & = & 
- \sum_{i=1}^n  \left(\left \langle R_\lambda|\CL_{E_i} (hR_\lambda) \rangle \langle E_i\right|  +
\left\langle R_\lambda |E_i \rangle \langle \CL_{E_i}(hR_\lambda)\right|\right) \\
& = & - \sum_{i=1}^n  \left \langle R_\lambda| dh(E_i) R_\lambda + h \CL_{E_i} R_\lambda \rangle \langle E_i\right|\\
& = & - h \sum_{i=1}^n  \left \langle R_\lambda |\CL_{E_i} R_\lambda \rangle \langle E_i\right|  
- \sum_{i=1}^n  \left\langle R_\lambda |R_\lambda \rangle \langle dh(E_i) E_i\right| \\
& = & - \sum_{i=1}^n  \langle dh(E_i) E_i|
\eeastar
where the last equality follows from $\langle R_\lambda |\CL_{E_i} R_\lambda \rangle = 0$ which 
holds the case because $\CL_{E_i} R_\lambda \in \xi$.

Similar computation gives rise to
$$
\sum_{i=1} \CL_{hR_\lambda}(|JE_i\rangle \langle JE_i|) 
=- \sum_{i=1}^n \left|dh(JE_i) JE_i \rangle \langle R_\lambda \right|.
$$
Summing up the two, we have derived
\beastar
&{}& 
\sum_{i=1} \CL_{hR_\lambda}(|E_i\rangle \langle E_i|)(\lambda) + \sum_{i=1} \CL_{hR_\lambda}(|J E_i\rangle \langle JE_i|) \\
& = & - \sum_{i=1}^n  dh(E_i) \langle  E_i| + dh(JE_i)  \langle JE_i|= - (dh)^\pi
\eeastar
which finishes the proof.

\def\cprime{$'$}
\providecommand{\bysame}{\leavevmode\hbox to3em{\hrulefill}\thinspace}
\providecommand{\MR}{\relax\ifhmode\unskip\space\fi MR }
\providecommand{\MRhref}[2]{%
  \href{http://www.ams.org/mathscinet-getitem?mr=#1}{#2}
}
\providecommand{\href}[2]{#2}


\begin{thebibliography}{LOTV18}

\bibitem[CS16]{casals-spacil}
Roger Casals and Oldrich Sp\'acil, \emph{Chern-{Weil} theory and the group of
  strict contactomorphisms}, J. Topol. Anal. \textbf{8} (2016), no.~1, 59--87.

\bibitem[Dir58]{dirac}
P.~A.~M. Dirac, \emph{The principles of quantum mechanics}, International
  Series of Monographs on Physics, Oxford University Press, 1958, 4th edition,
  xii + 314 pp.

\bibitem[dLLV19]{dMV}
Manuel de~Le\'on and Manuel Lainz~Valc\'azar, \emph{Contact {H}amiltonian
  systems}, J. Math. Phys. \textbf{60} (2019), no.~10, 102902, 18 pp.

\bibitem[DO]{do-oh:reduction}
Hyun-Seok Do and Y.-G. Oh, \emph{Thermodynamic reduction of contact dynamics},
  2024, arXiv:2412.19319.

\bibitem[EM70]{ebin-marsden}
David~G. Ebin and Jerrold Marsden, \emph{Groups of diffeomorphisms and the
  motion of an incompressible fluid}, Ann. of Math. (2) \textbf{92} (1970),
  102--163.

\bibitem[Eps84]{epstein:commutators}
D.B.A. Epstein, \emph{Commutators of $c^\infty$-diffeomorphisms. {A}ppendix to:
  `` a curious remark concerning the geometric transfer map'' by {J}ohn {N}.
  {M}ather, [{C}omment. {M}ath. {H}elv. \textbf{59} (1984), no. 1, 86--110.]},
  Comment. Math. Helv. \textbf{59} (1984), no.~1, 111--122.

\bibitem[Flo88]{floer:unregularized}
Andreas Floer, \emph{The unregularized gradient flow of the symplectic action},
  Comm. Pure Appl. Math. \textbf{41} (1988), no.~6, 775--813.

\bibitem[FU84]{freed-uhlen}
D.~Freed and K.~Uhlenbeck, \emph{Instantons and {F}our-{M}anifolds}, MSRI
  Publ., vol.~1, Srpringer-Verlag, New York, 1984.

\bibitem[LOTV18]{LOTV}
Hong~Van Le, Y.-G. Oh, Alfonso Tortorella, and Luca Vitagliano,
  \emph{Deformations of coisotropic submanifolds in {J}acobi manifolds}, J.
  Symplectic Geom. \textbf{16} (2018), no.~4, 1051--1116.

\bibitem[Lyc77]{lychagin}
V.~V. Lychagin, \emph{Sufficient orbits of a group of contact diffeomorphisms},
  Mat. USSR-Sb. (N.S.) \textbf{33} (1977), no.~2, 223--242.

\bibitem[Mat74]{mather}
John~N. Mather, \emph{Commutators of diffeomorphisms}, Comment. Math. Helv.
  \textbf{49} (1974), 512--528.

\bibitem[Mat75]{mather2}
\bysame, \emph{Commutators of diffeomorphisms. {II}}, Comment. Math. Helv.
  \textbf{50} (1975), 33--40.

\bibitem[Oha]{oh:nonprojectible}
Y.-G. Oh, \emph{Foliation de {R}ham cohomology of generic {R}eeb foliations},
  preprint 2025, arXiv:2504.16453(v2).

\bibitem[Ohb]{oh:simplicity}
\bysame, \emph{Simplicity of contactomorphism group of finite regularty},
  preprint 2014, arXiv:2403.18261.

\bibitem[Oh96]{oh:fredholm}
\bysame, \emph{Fredholm theory of holomorphic discs under the perturbation of
  boundary conditions}, Math. Z. \textbf{222} (1996), no.~3, 505--520.

\bibitem[Oh21]{oh:contacton-Legendrian-bdy}
\bysame, \emph{Contact {H}amiltonian dynamics and perturbed contact instantons
  with {L}egendrian boundary condition}, preprint, arXiv:2103.15390(v2), 2021.

\bibitem[OP05]{oh-park:coisotropic}
Y.-G. Oh and Jae-Suk Park, \emph{Deformations of coisotropic submanifolds and
  strong homotopy {L}ie algebroids}, Invent. Math. \textbf{161} (2005), no.~2,
  287--360. \MR{2180451 (2006g:53152)}

\bibitem[OW14]{oh-wang:connection}
Y.-G. Oh and R.~Wang, \emph{Canonical connection on contact manifolds}, Real
  and Complex Submanifolds, Springer Proceedings in Mathematics \& Statistics,
  vol. 106, 2014, (arXiv:1212.4817 in its full version), pp.~43--63.

\bibitem[OW18]{oh-wang:CR-map1}
\bysame, \emph{Analysis of contact {C}auchy-{R}iemann maps {I}: A priori
  {$C^k$} estimates and asymptotic convergence}, Osaka J. Math. \textbf{55}
  (2018), no.~4, 647--679.

\bibitem[Tsu08]{tsuboi:simplicity}
T.~Tsuboi, \emph{\em on the simplicity of the group of contactomorphisms}, Adv.
  Stud. Pure Math., vol.~52, pp.~491--504, Mathematical Society of Japan,
  Tokyo, 2008.

\bibitem[Wen]{wendl:lecture}
Chris Wendl, \emph{Lectures on {S}ymplectic {F}ield {T}heory}, unpublished book
  manuscript, 2016.

\bibitem[Wen23]{wendl:super-rigidity}
\bysame, \emph{Transversality and super-rigidity for multiply covered
  holomorphic curves}, Ann. of Math. (2) \textbf{198} (2023), no.~1, 93--230.

\end{thebibliography}
\end{document}